\documentclass[
reqno,11pt,oneside]{amsart}
\usepackage{amsmath,amssymb,amsfonts}
\usepackage{tikz-cd}
\usepackage{amsbsy}
\usepackage{amscd}
\usetikzlibrary{matrix,arrows,decorations.pathmorphing}
\usepackage{graphicx}
\usepackage{float}
\usepackage{enumitem}
\usepackage{setspace}
\usepackage[margin=1.0in]{geometry}
\usepackage[all]{xy}

\usepackage{latexsym}

\usepackage{color}

\usepackage
{hyperref}
\hypersetup{%
  colorlinks,
  bookmarksopen,
  bookmarksnumbered,
  citecolor=blue,
  linkcolor=black,
  pdfstartview=FitH,
}

\newtheorem{theorem}{Theorem}[section]
\newtheorem{lemma}[theorem]{Lemma}
\newtheorem{proposition}[theorem]{Proposition}

\theoremstyle{definition}
\newtheorem{example}[theorem]{Example}
\newtheorem{definition}[theorem]{Definition}
\newtheorem{remark}[theorem]{Remark}

\DeclareMathOperator{\id}{id}
\DeclareMathOperator{\tr}{tr}

\newcommand{\cl}[1]{\mathcal{#1}}
\newcommand{\bb}[1]{\mathbb{#1}}

\newcommand{\bC}{\bb C}

\newcommand{\wh}[1]{\widehat{#1}}

\DeclareMathOperator{\Ad}{Ad}

\let\norm\nor
\newcommand{\la}{\langle}
\newcommand{\ra}{\rangle}
\newcommand{\vphi}{\varphi}

\newcommand{\om}{\omega}

\newcommand{\lm}{\lambda}
\newcommand{\Ltau}{\Lambda_\tau}
\newcommand{\BH}{\mathcal{B}(H)}

\newcommand{\N}{\mathbb{N}}
\newcommand{\Z}{\mathbb{Z}}
\newcommand{\R}{\mathbb{R}}
\let\C\outer\renewcommand{\C}{\mathbb{C}}
\newcommand{\ten}{\otimes}
\newcommand{\oten}{\overline{\otimes}}

\let\phi\varphi

\let\epsilon\varepsilon

\def\ket#1{| #1 \rangle}
\def\bra#1{\langle #1 |}
\def\kb#1#2{|#1\rangle\!\langle #2 |}

\newcommand{\mc}[1]{\mathcal{#1}}
\newcommand{\BLTM}{\mc{B}(L^2(M,\tau))}

\let\mpar\marginpar
\def\marginpar#1{\mpar{#1}\ignorespaces}

\begin{document}

\title[]{Quantum teleportation in the commuting operator framework}

\author[A. Conlon]{Alexandre Conlon$^{1}$}

\author[J. Crann]{Jason Crann$^{1}$}
\address{$^1$School of Mathematics \& Statistics, Carleton University, Ottawa, ON, Canada H1S 5B6}
\email{alexandreconlon@cmail.carleton.ca}
\email{jasoncrann@cunet.carleton.ca}

\author[D. W. Kribs]{David W. Kribs$^{2,3}$}
\address{$^2$Department of Mathematics \& Statistics, University of Guelph, Guelph, ON, Canada N1G 2W1}
\address{$^3$Institute for Quantum Computing, University of Waterloo, Waterloo, ON, Canada N2L 3G1}
\email{dkribs@uoguelph.ca}

\author[R. H. Levene]{Rupert H. Levene$^{4,5}$}
\address{$^4$School of Mathematics and Statistics, University College Dublin, Belfield, Dublin 4, Ireland}
 \address{$^5$Centre for Quantum Engineering, Science, and Technology, University College Dublin, Belfield, Dublin 4, Ireland}

\email{rupert.levene@ucd.ie}




\begin{abstract} We introduce a notion of teleportation scheme between subalgebras of semi-finite von Neumann algebras in the commuting operator model of locality. Using techniques from subfactor theory, we present unbiased teleportation schemes for relative commutants $N'\cap M$  of a large class of finite-index inclusions $N\subseteq M$ of tracial von Neumann algebras, where the unbiased condition means that no information about the teleported observables are contained in the classical communication sent between the parties. For a large class of subalgebras $N$ of matrix algebras $M_n(\C)$, including those relevant to hybrid classical/quantum codes, we show that any tight teleportation scheme for $N$ necessarily arises from an orthonormal unitary Pimsner-Popa basis of $M_n(\C)$ over $N'$, generalising work of Werner \cite{w}. Combining our techniques with those of Brannan-Ganesan-Harris \cite{bgh}, we compute quantum chromatic numbers for a variety of quantum graphs arising from finite-dimensional inclusions $N\subseteq M$.
\end{abstract}

\maketitle

\section{Introduction}

Quantum teleportation \cite{betal1}, the transfer of qubits between two separated parties using preshared entanglement and local operations and classical communication, is a quintessential protocol in quantum information. It and its variants are used in a multitude of scenarios, including quantum error correction \cite{betal2}, quantum cryptography \cite{gr} and universal quantum computation \cite{gc}. 

Mathematically, the protocol involves three single qubit quantum systems, $A_0\otimes A_1 \otimes B = \mathbb C^2 \otimes \mathbb C^2 \otimes \mathbb C^2$, with Alice ($A$) having access to the first two systems and Bob ($B$) having access to the third. The composite system $A_1 B$ is entangled in the (maximally entangled) Bell state $\ket{\beta_{00}} = \frac{1}{\sqrt{2}}(\ket{00} + \ket{11})$. The entire two qubit Bell basis for $A_1B$ is then obtained as $\ket{\beta_{ij}} = (1 \otimes X^i Z^j)\ket{\beta_{00}}$, for $0 \leq i,j \leq 1$ and $X,Z$ denoting the usual single qubit Pauli operators (in this case acting on the third qubit). One can then verify that for an arbitrary single qubit $\ket{\psi}\in \mathbb C^2$, we have 
\[
\ket{\psi}\ket{\beta_{00}} = \frac12 \big(  \ket{\beta_{00}} \ket{\psi} +   \ket{\beta_{01}}(X \ket{\psi})  +  \ket{\beta_{10}}(Z\ket{\psi})  +  \ket{\beta_{11}} (XZ\ket{\psi})  \big). 
\]
Thus, $A$ proceeds by measuring the first two qubit composite system $A_0A_1$ in its Bell basis $\{\ket{\beta_{ij}}\}$, which is given as a quantum measurement on the combined three qubit system by the family of projection operators $P_{ij}\otimes 1$, $0 \leq i,j \leq 1$, with $P_{ij} = \kb{\beta_{ij}}{\beta_{ij}}$ on $A_0A_1$ and $1$ the identity operator on $B$. The party $A$ then communicates the result $(i,j)$ of this measurement to $B$, who then implements the unitary reversal $(X^i Z^j)^*$ on the third system to obtain the state $\ket{\psi}$, and this completes the protocol.

Using the maximially entangled state $|\psi_n\ra:=\frac{1}{\sqrt{n}}\sum_{i=0}^{n-1}|ii\ra\in\C^n\ten\C^n$ and the generalised Pauli $X$ and $Z$ operators on $\C^n$, the procedure generalises verbatim to states in $\C^n$. This latter protocol was put into the larger context of \textit{teleportation schemes} by Werner \cite{w}, which allowed for broader possible implementations by the parties. Specifically, a teleportation scheme for $\C^n$ consists of a triple $(\om,\{F_i\}_{i\in I}, \{T_i\}_{i\in I})$ where $\om$ is a density on $\C^n\ten \C^n$ (entangled resource state), $\{F_i\}_{i\in I}$ is a positive operator-valued measure (POVM) on $\C^n\ten\C^n$ (Alice's measurement system) and $\{T_i\}_{i\in I}$ are unital completely positive (UCP) maps on $M_n(\C)$ (Bob's quantum channels in the Heisenberg picture) for which
$$\tr(\rho B)=\sum_{i\in I} \tr((\rho\ten \om)(F_i\ten T_i(B))), \ \ \ \forall \ \rho, B\in M_n(\C).$$
The scheme $(\om,\{F_i\}_{i\in I}, \{T_i\}_{i\in I})$ is \textit{tight} \cite{w} if $|I|=n^2$, that is, the amount of classical signals communicated coincides with the dimension of the algebra to be teleported: in the Heisenberg picture, Bob's observable algebra ($M_n(\C)$) is teleported to Alice's local observable algebra. Werner established a correspondence between tight teleportation schemes for $M_n(\C)$ and unitary error bases of $M_n(\C)$ \cite{w}, that is, orthonormal bases of unitaries with respect to the (normalised) Hilbert-Schmidt inner product (a prime example given by the generalised Pauli operators; see Example \ref{ex:1}). 

Rather than teleportation of a full system, one can imagine scenarios in which it is desirable to teleport quantum information encoded into subsystems of a full system or even hybrid forms of classical and quantum information. This could arise for instance with subsystem codes used in quantum algorithms and fault tolerant quantum computing architectures \cite{aliferis2007subsystem,bacon2006operator,klappenecker2008clifford,kribs2005unified,oreshkov2009fault,poulin2005stabilizer} or hybrid codes used for the simultaneous transmission of classical and quantum information in communication schemes \cite{bkk1,bkk1a,cao2021higher,devetak2005capacity,grassl2017codes,kremsky2008classical,kuperberg2003capacity,nemec2018hybrid,nemec2021infinite}. Such scenarios admit convenient mathematical descriptions in the Heisenberg picture, and as such, can be studied through generalisations of teleportation schemes to the commuting operator framework, in which locality is modelled by commuting algebras of observables~\cite{hk}. Indeed, the mathematical origins of quantum theory \cite{vN} together with recent advances in non-local games (e.g., \cite{dpp,musath,jietal,jungeetal,ozawa,slofstra}) and the increasing number of connections between quantum information and quantum field theory (e.g., \cite{adh,cnn,gk,harlow,hs}) continue to motivate the study/extension of central results in quantum information to the commuting operator framework. Examples include entropy theory (e.g., \cite{bfs,ggpp,gjl,hiaif1,hiaif2,lx}), quantum error correction (e.g., \cite{bkk2,cklt2}), the theory of local operations (e.g., \cite{cklt,gjl2,vw}) and entanglement in quantum field theory (e.g., \cite{hs} and the references therein).

In this paper we continue this line of work by introducing a general notion of teleportation scheme in the setting of semi-finite von Neumann algebras and studying analogues of tightness and related properties in this setting. Our main examples utilize subfactor theory  and deepen the connection between standard teleportation of observables in $M_n(\C)$ and Jones' basic construction for the inclusion $\C\subseteq M_n(\C)$ (see, e.g., \cite{huang,zhang} and Examples \ref{ex:3.3} and \ref{ex:3.4}). In particular, in section 3 we establish unbiased teleportation protocols for relative commutants $N'\cap M$ of (finite-index) inclusions $N\subseteq M$ which admit orthonormal Pimsner-Popa bases in the unitary normaliser of $N$. See Theorem \ref{t:tele} for details.

Restricting attention to the tripartite system $M_n(\C)\ten M_n(\C)\ten M_n(\C)$, in section 4 we generalise Werner's characterization of tight teleportation schemes \cite{w}, showing that, for a large class of (unital $*$-)subalgebras $N\subseteq M_n(\C)$, any tight (and faithful) teleportation scheme of $N$ between the first and third subsystems necessarily arises from a unitary orthonormal Pimsner-Popa basis for $M_n(\C)$ over $N'$. See Theorem \ref{t:Werner} for a precise statement. Our result applies, in particular, to homogeneous subalgebras $N$, which model hybrid classical/quantum codes in the Heisenberg picture \cite{bkk1,cao2021higher,grassl2017codes,nemec2018hybrid,nemec2021infinite}.

It is known that unitary error bases generate quantum-to-classical graph homomorphisms \cite{bgh,mrv}, and therefore can be used to estimate/calculate various chromatic numbers of quantum graphs~\cite{bgh}. Based on our generalised teleportation schemes and techniques from \cite{bgh}, in section 5 we calculate various chromatic numbers for quantum graphs arising from inclusions $N\subseteq M$ on a finite-dimensional Hilbert space $H$. In particular, when $N$ is a factor, we show that 
$$\chi_q(N',M,\mc{B}(H))=\chi_{qc}(N',M,\mc{B}(H))=[M:N],$$
where $[M:N]$ is the index of $N$ and $M$, and $\chi_q$ and $\chi_{qc}$ are the quantum and quantum commuting chromatic numbers of the quantum graph $(N',M,\mc{B}(H))$, respectively. This generalises the case $N=\C$ established in \cite{bgh}. Also, when the inclusion $N\subseteq M$ admits an orthonormal Pimsner-Popa basis in the unitary normaliser of $N$, we show that 
$$\chi_{loc}(M,N',\BLTM)=\chi_{q}(M,N',\BLTM)=\chi_{qc}(M,N',\BLTM)=[M:N],$$
where $\chi_{loc}$ is the local chromatic number. 

Several natural lines of investigation are suggested by this work, including connections with weak Hopf $C^*$-algebras and the structure of depth-2 subfactors \cite{nv}, as well as diagrammatic representations of teleportation in monoidal categories \cite{ac,jlw,lwj}. The final outlook section elaborates on these and other connections left for future work.

\section{Preliminaries}

In this section we outline relevant preliminaries from the theory of tracial von Neumann algebras, including Jones' basic construction, Pimsner-Popa bases, and notions of entanglement in the commuting operator framework.

\subsection{Jones' Basic Construction}
Let $M$ be a tracial von Neumann algebra, that is, a finite von Neumann algebra with fixed normal faithful tracial state $\tau$. Inclusions of von Neumann subalgebras $N\subseteq M$ will always be assumed unital. 

The GNS construction of $(M,\tau)$ yields the Hilbert space $L^2(M,\tau)$, the GNS map $\Lambda_\tau:M\rightarrow L^2(M,\tau)$ and the (faithful) representation $\pi_\tau:M\rightarrow\cl B(L^2(M,\tau))$, where
$$\pi_\tau(x)\Ltau(y)=\Ltau(xy), \ \ \ x,y\in M.$$
We often simply write $x$ for $\pi_\tau(x)$ (and $M$ for $\pi_\tau(M)$) when convenient. The inner product on $L^2(M,\tau)$ satisfies
$$\la\Ltau(x),\Ltau(y)\ra=\tau(y^*x), \ \ \ x,y\in M.$$
The adjoint operation in $M$ yields a conjugate linear isometry $J$ on $L^2(M,\tau)$ via
$$J\Ltau(x)=\Ltau(x^*), \ \ \ x\in M.$$
The associated (anti)representation of right multiplication $\pi_\tau^r:M\rightarrow\cl B(L^2(M,\tau))$ is given by
$$\pi_\tau^r(x)\Ltau(y)=\Ltau(yx), \ \ \ x,y\in M.$$
One easily sees that $\pi_\tau^r(x)=J\pi_\tau(x)^*J$ for every $x$ in $M$, and that $\pi_\tau^r$ is $*$-preserving by traciality. It follows that $\pi_\tau(M)'=\pi_\tau^r(M)=J\pi_\tau(M)J$ (see, e.g., \cite[Theorem  V.2.22]{t1}), that is, $M'=JMJ$ in $\cl B(L^2(M,\tau))$.

For a von Neumann subalgebra $N\subseteq M$, let $L^2(N,\tau)=\overline{\Ltau(N)}$ be the associated closed subspace of $L^2(M,\tau)$. The orthogonal projection $e_N$ onto $L^2(N,\tau)$ induces the unique $\tau$-preserving faithful normal conditional expectation $E_N:M\rightarrow N$ via
$$e_N\Ltau(x)=\Ltau(E_N(x)),\ \ \ x\in M.$$
(See \cite[Theorem IX.4.2]{t2} for a more general result.) The projection $e_N$ is commonly known as the \textit{Jones projection} for the inclusion $N\subseteq M$. The von Neumann subalgebra $M_1:=\la M,e_N\ra$ of $\cl B(L^2(M,\tau))$ generated by $M$ and $e_N$ is the result of the basic construction of the inclusion $N\subseteq M$. We list some standard facts (see, e.g., \cite[\S3.1]{js}):
\begin{enumerate}
\item $e_N x e_N = E_N(x)e_N$, $x\in M$;
\item $E_N(axb)=aE_N(x)b$, $a,b\in N$, $x\in M$;
\item $e_N\in N'$;
\item $Je_N=e_NJ$;
\item $M_1=JN'J$.
\end{enumerate}
Property (5) together with the equality $M'=JMJ$ imply that $J(N'\cap M)J=M'\cap M_1$. It follows that the map
$$\gamma_0:=\pi^r_\tau|_{N'\cap M}:N'\cap M\ni x\mapsto Jx^*J\in M'\cap M_1$$
is an anti-isomorphism. 

The algebra $M_1$ has a canonical faithful semi-finite normal trace $\mathrm{tr}_1$ determined by $\mathrm{tr}_1(xe_Ny)=\tau(xy)$, $x,y\in M$ \cite[\S 1.1.2]{p}. The trace $\tau$ is \textit{Markov} for the inclusion $N\subseteq M$ if $\mathrm{tr}_1$ is finite, and $\tau_1:=\mathrm{tr}_1(1)^{-1}\mathrm{tr}_1$ has $\tau_1|_M=\tau$. In this case, we may iterate the basic construction to obtain $M\subseteq M_1\subseteq M_2$, where $M_2$ is the von Neumann subalgebra of $\cl{B}(L^2(M_1,\tau_1))$ generated by $M_1$ and $e_M$, the Jones projection for the inclusion $M\subseteq M_1$. As above, $M_2=J_{\tau_1}M'J_{\tau_1}$, and 
$$\gamma_1:=\pi^r_{\tau_1}|_{M'\cap M_1}:M'\cap M_1\rightarrow M_1'\cap M_2$$
is an anti-isomorphism. The composition 
$$\Gamma:=\gamma_1\circ\gamma_0:N'\cap M\rightarrow M_1'\cap M_2$$ 
is therefore a $*$-isomorphism, known as the \textit{canonical shift} (see e.g., \cite[Proposition 2.22]{jp}). One may continue to iterate the basic construction, yielding an increasing sequence of finite von Neumann algebras $N\subseteq M\subseteq M_1\subseteq \cdots$ known as the Jones tower of the inclusion $N\subseteq M$. In this paper we will only be concerned with the first two iterations. 

\subsection{Bases} Let $N\subseteq M$ be an inclusion of finite von Neumann algebras. A finite subset $B=\{\lm_i\mid i=1,...,d\}\subseteq M$ is a (left) \textit{Pimsner--Popa basis}, or simply \textit{basis}, for $M$ over $N$ if either of the following equivalent conditions hold:
\begin{enumerate}
\item $\sum_{i=1}^d \lm_i^*e_N\lm_i = 1$;
\item $x=\sum_{i=1}^d E_N(x\lm_i^*)\lm_i$ for all $x\in M$.
\end{enumerate}
When $E_N(\lm_i\lm_j^*)=\delta_{i,j}1$, we say that $\{\lm_i\}_{i=1}^d$ is an \textit{orthonormal basis} of $M$ over $N$ (compare with \cite{p}, wherein orthonormality allows $E_N(\lm_i\lm_i^*)$ to be a projection in $N$). In this case, $B$ forms an orthonormal basis of $M$ as a (left) Hilbert $N$-module, with respect to the $N$-valued inner product $\la x,y\ra_N=E_N(xy^*)$. This notion of basis was introduced in \cite{pp} in the setting of II$_1$-factors, and was later generalised (see, e.g., \cite{p,wat}).

Following the terminology of \cite{jp}, we call an inclusion $N\subseteq M$ \textit{strongly Markov} if the trace $\tau$ is Markov and there exists a finite Pimsner-Popa basis for $M$ over $N$. In this case, the element
$$\sum_{i}\lm_i^*\lm_i\in \R^+1$$
is independent of the Pimsner-Popa basis. Indeed, by \cite[\S 1.1.4]{p} $E_M(e_N)=\alpha1$ for some scalar $\alpha>0$ (where $E_M$ is the $\tau_1$-preserving conditional expectation $M_1\rightarrow M$), so that
$$\sum_{i}\lm_i^*\lm_i=\alpha^{-1}\sum_iE_M(\lm_i^*e_N\lm_i)=\alpha^{-1}1.$$
The associated scalar $\alpha^{-1}$ is the (Watatani) index $[M:N]$ of $M$ in $N$ \cite{wat}. Thus, we have
\begin{equation}\label{e:Markov} E_M(e_N)=[M:N]^{-1}1.\end{equation}
It follows that
\begin{equation}\label{e:TL} e_Ne_Me_N=[M:N]^{-1}e_N, \ \ \ e_Me_Ne_M=[M:N]^{-1}e_M, \ \ \ [M:N]=[M_1:M].\end{equation}
For example, a finite-index inclusion $N\subseteq M$ of II$_1$ factors is strongly Markov with  the index coinciding with the Jones index \cite{j}. Also, a connected (i.e., $Z(N)\cap Z(M)=\C$) inclusion $N\subseteq M$ of finite-dimensional $C^*$-algebras is strongly Markov \cite[Corollary 3.2.5]{js} and the index coincides with $\norm{\Lambda_N^M}^2$ \cite[Proposition 3.3.2]{js}, the square of the (operator) norm of the inclusion matrix $\Lambda_N^M$.

\begin{example}\label{ex:1} Let $N=\C\subseteq M$. Then $E_N=\tau(\cdot)1$ is the ``completely depolarising channel''. The Jones projection satisfies
$$e_N\Ltau(x)=\Ltau(E_N(x))=\tau(x)\Ltau(1)=\Ltau(1)\Ltau(1)^*\Ltau(x), \ \ \ x\in M.$$
Hence, $e_N=\Ltau(1)\Ltau(1)^*$. 

When $M=M_n(\C)$, $\Lambda_{\tau_n}(1)=n^{-1/2}\sum_{i=0}^{n-1}|ii\ra=:\psi_n$ is the maximally entangled state. In this case, any orthonormal basis of $M_n(\C)$ with respect to $\tau_n$ (i.e., the normalised Hilbert--Schmidt inner product) will form an orthonormal Pimsner--Popa basis for the inclusion $\C\subseteq M_n(\C)$. A natural choice is the image of the Weyl representation
$$W:\Z_n\times\Z_n\ni(k,l)\mapsto V^lU^k\in M_n(\C),$$
where $U$ and $V$ are the translation and multiplication operators associated with the standard basis $\{\ket{k}\}_{k\in \Z_n}$ of $\C^n$ (also known as generalised Pauli operators): 
\begin{equation}\label{e:U}U|k\ra=|k+1\ra, \ \ \ V|k\ra=e^{2\pi ik/n}|k\ra, \ \ \ k\in\Z_n.\end{equation}
One easily verifies that 
$$\tau_n(W(z')^*W(z))=\delta_{z',z}, \ \ \ z,z'\in\Z_n^2.$$

\end{example}

\begin{example}\label{ex:2} Let $N=\ell^\infty_n$ (diagonals) inside $M=M_n(\C)$. Then the translation operators $\{U^k\mid k\in\Z_n\}$ with $U$ defined as in Example \ref{ex:1} form an orthonormal basis for $M$ over $N$. Indeed, the decomposition
$$x=\sum_{k\in\Z_n} E_N(x(U^k)^*)U^k$$
corresponds to breaking $x$ into the sum of its diagonals. 

Another basis of $M_n(\C)$ over $\ell^\infty_n$ is $\{\frac{1}{\sqrt{n}}|\chi\ra\la\chi|\mid \chi\in\wh{\Z_n}\}$. This ``character basis'' of $M_n(\C)$ over $\ell^\infty_n$ has the nice property that each $|\chi\ra$ (suitably normalized) acts as a trace vector for $\ell^\infty_n$. 
\end{example}

\begin{example}\label{ex:3} The first part of Example \ref{ex:2} generalises naturally to crossed product inclusions. Let $G$ be a finite group acting by $\tau$-preserving automorphisms $\alpha_s$, $s\in G$, on a tracial von Neumann algebra $(M,\tau)$. The action induces a unital, injective $*$-homomorphism 
$$\alpha:M\ni x\mapsto (s\mapsto \alpha_{s^{-1}}(x))\in \ell^\infty(G,M),$$
where $\ell^\infty(G,M)=\ell^\infty(G)\ten M$ denotes the (bounded) $M$-valued functions on $G$. The crossed product $G\ltimes M$ is the von Neumann subalgebra of $\cl B(\ell^2(G))\oten M$ generated by $\alpha(M)$ and the (amplified) left regular representation $\lm(G)\ten 1$ under the covariance relations
$$(\lm_s\ten 1)\alpha(x)(\lm_s\ten 1)^*=\alpha(\alpha_s(x)), \ \ \ x\in M, \ s\in G,$$
where $\oten$ denotes the von Neumann algebra tensor product. The inclusion $\alpha(M)\subseteq G\ltimes M$ admits a canonical conditional expectation $E:G\ltimes M\rightarrow \alpha(M)$ satisfying
$$E\bigg(\sum_{s\in G}\alpha(x_s)(\lm_s\ten 1)\bigg)=\alpha(x_e), $$
for any collection $x_s$, $s\in G$ in $M$. In this case, $\{(\lm_s\ten 1)\mid s\in G\}$ is a basis of $G\ltimes M$ over $\alpha(M)$ via the usual ``Fourier'' decomposition 
$$X=\sum_{s\in G} E(X(\lm_s\ten 1)^*)(\lm_s\ten 1), \ \ \ X\in G\ltimes M.$$
\end{example}

\subsection{Entanglement}\label{s:ent} The entanglement of $\psi_n=n^{-1/2}\sum_{i=0}^{n-1}|ii\ra\in \C^n\ten\C^n$ is expressed through the relation 
\begin{equation}\label{eq:ME_n}\pi_{\tau_n}(x)\psi_n=(x\ten 1)\psi_n=(1\ten x^t)\psi_n=\gamma_0(x)\psi_n, \ \ \ x\in M_n(\C).\end{equation}
From a mathematical perspective, the heart of the standard teleportation protocol (see below) is the following double application of (\ref{eq:ME_n}):
\begin{equation}\label{eq:ME_n(2)}(|\psi_n\ra\la\psi_n|\ten 1)(x\ten 1\ten 1)(1\ten|\psi_n\ra\la\psi_n|)=(|\psi_n\ra\la\psi_n|\ten 1)(1\ten 1\ten x)(1\ten|\psi_n\ra\la\psi_n|), \ \ \ x\in M_n(\C).\end{equation}
Viewing $\psi_n$, or rather its associated density, as the Jones projection for $\C\subseteq M_n(\C)$, these manifestations of entanglement generalize naturally, as is well-known (see \cite[Lemma 2.4]{bisch}, for example).

\begin{lemma}\label{l:me} Let $N\subseteq M$ be an inclusion of finite von Neumann algebras. For any $x\in N'\cap M$, 
$$xe_N=\pi^r_\tau(x)e_N=\gamma_0(x)e_N.$$
\end{lemma}

\begin{proof} For any $y\in M$,
\begin{align*} xe_N\Ltau(y)&=x\Ltau(E_N(y))=\Ltau(xE_N(y))=\Ltau(E_N(y)x)\\
&=\pi_\tau^r(x)\Ltau(E_N(y))=\gamma_0(x)e_N\Ltau(y).
\end{align*}
\end{proof}

\begin{lemma}\label{l:me2} Let $N\subseteq M$ be a strongly Markov inclusion. For any $x\in N'\cap M$, 
$$e_Nxe_M=e_N\Gamma(x)e_M.$$
\end{lemma}

\begin{proof} Taking the adjoint of the relation from Lemma \ref{l:me} we have $e_Nx=e_N\gamma_0(x)$ with $\gamma_0(x)\in M'\cap M_1$. Applying Lemma \ref{l:me} to the inclusion $M\subseteq M_1$, we have $\gamma_0(x)e_M=\gamma_1(\gamma_0(x))e_M=\Gamma(x)e_M$. The result follows.
\end{proof}

Lemma \ref{l:me} implies that any unit vector $\psi\in L^2(N,\tau)$ is a perfectly correlated/EPR state with respect to the commuting algebras $N'\cap M$ and $M'\cap M_1$, meaning that any self-adjoint $x\in N'\cap M$ has an ``EPR double'' $x'\in M'\cap M_1$ for which
$$\la (x-x')^2\psi,\psi\ra=0.$$
Indeed, by Lemma~\ref{l:me}, for $x'=\gamma_0(x)$,
\[\la(x-x')^2\psi,\psi\ra=\|(x-x')\psi\|^2=\|(x-x')e_N\psi\|^2=0.\]
For details on perfect correlation see \cite{av} for the type I case and \cite{w99} for the general von Neumann algebraic setting (both works of course building on the seminal paper \cite{epr} of Einstein--Podolski--Rosen). More generally, if a unitary $u$ belongs to the normaliser 
$$\cl N_M(N):=\{u\in \mc{U}(M)\mid u^*Nu=N\},$$
then $u(N'\cap M)u^*=N'\cap M$ and $u^*\psi$ is also an EPR state with respect to the same commuting algebras:
\begin{equation}\label{e:normal}\gamma_0(uxu^*)u^*\psi=u^*\gamma_0(uxu^*)\psi=u^*\gamma_0(uxu^*)e_N\psi=u^*(uxu^*)\psi=xu^*\psi, \ \ \ x\in N'\cap M.\end{equation}
In other words, the EPR double of $x\in N'\cap M$ relative to the state $u^*\psi$ is $\gamma_0(uxu^*)\in M'\cap M_1$. Moreover, the restricted vector state $\om_{u^*\psi}|_{N'\cap M}$ is tracial, which is often viewed as a form of maximal entanglement in the commuting operator framework (see e.g., \cite[\S V.A]{keylsw} or \cite[\S6]{cklt}). Explicitly, for $x,y\in N'\cap M$, we have
\begin{align*}\om_{u^*\psi}(xy)&=\la xyu^*\psi,u^*\psi\ra=\la x\gamma_0(uyu^*)u^*\psi,u^*\psi\ra\\
&=\la\gamma_0(uyu^*)xu^*\psi,u^*\psi\ra=\la xu^*\psi,\gamma_0(uy^*u^*)u^*\psi\ra\\
&=\la xu^*\psi,y^*u^*\psi\ra=\la yxu^*\psi,u^*\psi\ra.
\end{align*}       

\section{Teleportation Schemes for Semi-Finite von Neumann Algebras}

The standard teleportation protocol fits naturally into the framework of the basic construction for $\C\subseteq M_n(\C)$. Indeed, iterating the construction gives
$$\underbrace{\C}_{N}\subseteq \underbrace{M_n(\C)}_{M}\subseteq \underbrace{M_n(\C)\ten M_n(\C)}_{M_1}\subseteq \underbrace{M_n(\C)\ten M_n(\C)\ten M_n(\C)}_{M_2}.$$
From the operator algebraic perspective, the observable algebra of the global system of the protocol is $M_2$, while Alice and Bob's observable algebras are $M_1\cong M_n(\C)\ten M_n(\C)\ten 1$ and $M_1'\cap M_2\cong 1\ten 1\ten M_n(\C)$, respectively. The first Jones projection $e_N\in M_1$ is the (rank-1 projection onto the) maximally entangled state $\psi_n$, while the second Jones projection $e_M\in M_2$ is $1\ten|\psi_n\ra\la\psi_n|$ (since $M=M_n(\C)\ten1$ when viewed as a subalgebra of $M_1=M_n(\C)\ten M_n(\C)$). Note that $e_M$ is precisely the entangled resource shared by Alice and Bob. Let $B=\{u_i\}_{i=1}^{n^2}$ be any unitary orthonormal basis of $M_n(\C)$ with respect to the normalised Hilbert--Schmidt inner product. Then $B$ is a Pimsner--Popa basis for $\C\subseteq M_n(\C)$, and Alice's local measurement in the associated teleportation protocol is with respect to the projection-valued measure (PVM) $\{P_i:=(u_i^*\ten 1)e_N(u_i\ten 1)\}\subseteq M_1$. Bob's operations are local conjugations by the $u_i$ on the third tensor factor, i.e., automorphisms of $M_1'\cap M_2$.

In the Heisenberg picture, the (standard) teleporation identity associated to the unitary basis $B$ is 
$$x=\sum_{i=1}^{n^2} (\id\ten\tr\ten\tr)((1\ten \psi_{\tau_n}\psi_{\tau_n}^*)(P_i\ten u_ixu_i^*)), \ \ \ x\in M_n(\C).$$
From this perspective, Bob's observable algebra $M_1'\cap M_2$ is teleported to Alice's local algebra $M=M_n(\C)\ten 1\ten 1=N'\cap M$. Recasting the above identity inside $M_2$,
$$x\ten 1\ten 1=n^2\sum_{i=1}^{n^2} E_{M}(e_M(P_i\ten u_ixu_i^*))=\Gamma^{-1}(1\ten 1\ten x),$$
where $E_{M}=(\id\ten\tau_n\ten\tau_n)$ is the unique trace-preserving conditional expectation from $M_2$ onto $M$, and $\Gamma$ is the canonical shift. Thus, from this perspective, the celebrated teleportation identity is an LOCC implementation of (the inverse of) the canonical shift associated to the inclusion $\C\subseteq M_n(\C)$. This observation, together with the framework of \cite{w} motivated the definition to follow. We first recall the commuting model of local operations and classical communication (LOCC) recently developed in \cite{cklt}.

Given two commuting von Neumann subalgebras $A$ and $B$ of an ambient von Neumann algebra $M$, a (one-way, right) LOCC operation is a normal, unital completely positive (UCP) map $\Phi:M\rightarrow M$ of the form $\Phi=\sum_{i=1}^\infty S_i\circ T_i$ (point weak*-convergent), where $S_i:M\rightarrow M$ is a normal CP $B$-bimodule map satisfying $S_i(A)\subseteq A$, and $T_i:M\rightarrow M$ are normal UCP $A$-bimodule maps satisfying $T_i(B)\subseteq B$. (Note that the invariance conditions $S_i(A)\subseteq A$ and $T_i(B)\subseteq B$ were automatic in \cite{cklt} since they considered $M=\BH$, and $B=A'$, in which case bimodularity implies invariance.) We also recall that $A\vee B$ denotes the von Neumann algebra generated by $A\cup B$.

\begin{definition}\label{d:tele}
Let $A$ and $B$ be two commuting von Neumann subalgebras of a tracial von Neumann algebra $(M,\tau)$. Suppose that $A$ contains von Neumann subalgebras $A_0$, $A_1$, for which there exist anti-isomorphisms $\gamma_0:A_0\rightarrow A_1$ and $\gamma_1:A_1\rightarrow B$. A \textit{teleportation scheme for $A_0$ relative to $A,B\subseteq  A\vee B$} consists of the following: 
\begin{itemize}
\item[$\bullet$] a $\tau$-density operator $\omega$ in $A_0'\cap M$;
\item[$\bullet$] a collection $\{T_i\}_{i\in I}$ of normal UCP $A$-bimodule maps on $A\vee B$ for which $T_i(B)\subseteq B$, $i\in I$.
\item[$\bullet$] a POVM $\{ F_i \}_{i\in I}$ in $A$ such that $\sum_{i\in I}\Ad(F_i)\circ T_i$ is a one-way right LOCC map relative to $A,B\subseteq A\vee B$, and 
\begin{equation}
\sum_{i\in I} E_{A_0}(F_iT_i(\Gamma(a))\om)=a,  \ \ \ a\in A_0,\end{equation}
where $\Gamma:A_0\rightarrow B$ is the $*$-isomorphism $\Gamma=\gamma_1\circ\gamma_0$, and $E_{A_0}$ is the $\tau$-preserving conditional expectation from $M$ onto $A_0$, and the series converges in the weak* topology.
\end{itemize}
The scheme $(\om,\{T_i\}_{i\in I},\{F_i\}_{i\in I})$ is 
\begin{itemize}
\item \textit{faithful} if $\tau(F_i\rho\om)>0$ for all $i\in I$ and $\tau$-densities $\rho\in A_0$. 
\item \textit{minimal} if $\om\in A_1\vee B$ and $\{F_i\}_{i\in I}\subseteq A_0\vee A_1$. 
\end{itemize}
When $|I|<\infty$, the scheme $(\om,\{T_i\}_{i\in I},\{F_i\}_{i\in I})$ is 
\begin{itemize}
\item[$\bullet$] \textit{tight} if  $\dim(A_0)=|I|$, 
\item[$\bullet$] \textit{unbiased} if $\tau(F_i\rho\om)=|I|^{-1}$ for all $i\in I$ and $\tau$-densities $\rho\in A_0$.
\end{itemize}
\end{definition}

\begin{remark}${}$
\begin{enumerate}
\item As in \cite{w}, tightness means the amount of classical signals sent from Alice to Bob equals the dimension of the algebra to be teleported.
\item Minimality implies that the resource state $\omega$ lives in the ``smallest'' algebra possessing a density which entangles $A_1$ and $B$. Similar remark for Alice's POVM $\{F_i\}$.  
\item The unbiasedness property ensures that for any input state $\rho$ from $A_0$, Alice's local measurement result is uniformly random, so that no information about $\rho$ is contained in the classical information sent to Bob.
\item Definition \ref{d:tele} readily generalises to semi-finite von Neumann algebras $(M,\tau)$ with a normal conditional expectation onto $A_0$. For instance, when the restriction of the (normal semi-finite faithful) trace $\tau$ to $A_0$ is semi-finite (in which case there is a unique normal $\tau$-preserving conditional expectation from $M$ onto $A_0$ by Takesaki's theorem \cite[Theorem IX.4.2]{t2}). One can also envision similarly defined schemes beyond semi-finite von Neumann algebras, although we will not pursue them in this paper.
\end{enumerate}
\end{remark}

\begin{example}\label{ex:3.3} Let $M=M_n(\C)\ten M_n(\C)\ten M_n(\C)$, 
$$A_0=M_n(\C)\ten 1\ten 1, \ A_1=1\ten M_n(\C)\ten 1, \ B=1\ten 1\ten M_n(\C).$$
We recover the standard teleportation scheme with
\begin{itemize}
\item $\om=n^2(1\ten|\psi_n\ra\la\psi_n|)\in A_1\vee B$ (the factor $n^2$ is to ensure normalisation with respect to $\tau_n\ten\tau_n$);
\item $T_i=\mathrm{Ad}(1\ten 1\ten u_i)$, where $\{u_i\}_{i=1}^{n^2}$ is any unitary orthonormal basis of $M_n(\C)$;
\item $\{F_i=(u_i^*\ten 1)|\psi_n\ra\la\psi_n|(u_i\ten 1)\}_{i=1}^{n^2}\subseteq A_0\vee A_1$.
\end{itemize}
\end{example}

\begin{example}\label{ex:3.4} Recently, Huang studied teleportation in a II$_1$-factor setting, showing the existence of certain (finite-dimensional) matrix subalgebras and analogues of the standard teleportation protocol for them \cite{huang}. We sketch how his protocol fits into our general framework, referring the reader to \cite[\S IV]{huang} for details.

Let $M$ be a II$_1$-factor on $H$ with cyclic and separating trace vector $\psi\in H$, that is, $\tau(x)=\la x\psi,\psi\ra$, $x\in M$ (where $\tau$ is the unique trace of $M$). Let $N=M'$, and let $\gamma_1(x)=Jx^*J$ denote the canonical anti-isomorphism between $M$ and $N$, where $J$ is the modular conjugation associated to $(M,\psi)$. The global system is $N\oten\mc{B}(H)$, where Alice's observable algebra is $N\ten M$ and Bob's is $1\ten N$. For a natural number $n\in\N$, Huang takes a PVM $\{P_j\}_{j=0}^{n-1}$ in $M$ consisting of equivalent projections satisfying $\tau(P_j)=\frac{1}{n}$ for all $j$ (which exists in II$_1$-factors), and defines an associated family $\{V_{j,k}\mid j,k=0,...,n-1\}\subseteq M$ of partial isometries which form matrix units for a subalgebra $A_1$ of $M$ isomorphic to $M_n(\C)$ \cite[Equation (6)]{huang}. Let $W_{j,k}=\gamma(V_{k,j})$,  
$$\Psi_{j,k}=\frac{1}{N}\sum_{\mu,\nu}e^{\frac{2\pi i k(\nu-\mu)}{n}}W_{\mu+j,\nu+j}\ten V_{\mu,\nu}, \ \ \ \textnormal{and} \ \ \  U_{j,k}=\sum_{\mu,\nu=0}^{n-1}e^{\frac{2\pi i k\nu}{n}}\delta_{\mu,\nu+j}V_{\mu,\nu},$$
where addition of the indicies is modulo $n$. Then $\{\Psi_{j,k}\mid j,k=0,...,n-1\}$ is PVM in $N\ten M$ (Alice's algebra) and $U_{j,k}$ are unitaries in $M$. With $A_0=B=\mathrm{span}\{W_{j,k}\mid j,k=0,...,n-1\}\subseteq N$, 
$$\gamma_0:A_0\ten 1\ni a\ten 1\mapsto 1\ten \gamma(a)\in 1\ten A_1,$$
and $\Gamma=(\id\ten\gamma_1)\circ\gamma_0:A_0\ten 1\rightarrow 1\ten B$, one can show that
\begin{equation}\label{e:huang}\sum_{j,k=0}^{n-1}(\id\ten\om_{\psi})(\Psi_{j,k}(\id\ten\Ad(\gamma(U_{j,k})))(\Gamma(a)))=a, \ \ \ a\in A_0.\end{equation}
Thus, Huang's protocol fits into our framework for semi-finite von Neumann algebras viewing $(\id\ten\om_{\psi})$ as a normal conditional expectation $N\oten\mc{B}(H)\rightarrow N$ (and further composing with the unique normal trace-preserving conditional expectation $E_{A_0}:N\rightarrow A_0$ on either side of (\ref{e:huang})).
\end{example}

Given the discussion at the beginning of the section, it is natural to consider teleportation schemes for more general inclusions $N\subseteq M$ in connection with the basic construction, in particular, when $N$ and/or $M$ are not necessarily factors. We first examine the natural ``direct sum protocol'' when $N=\C$, and $M$ is finite-dimensional. 

Suppose $M\cong\bigoplus_{j=1}^m M_{n_j}(\C)$. The (unique) Markov trace $\tau$ for the inclusion $\C\subseteq M$ is 
$$\tau = \frac{1}{\dim M} \sum_{j=1}^m n_j\tr_{n_j},$$
where $\tr_{n_j}$ is the unnormalized trace on $M_{n_j}(\C)$. It follows that $L^2(M,\tau)\cong\bigoplus_{j=1}^m \C^{n_j}\ten\C^{n_j}$, and that $M$ is represented on $L^2(M,\tau)$ as $M=\bigoplus_{j=1}^m M_{n_j}(\C)\ten 1_{n_j}$. Note that since $N=\C$, we have $M_1=JN'J=\cl B(L^2(M,\tau))$. The anti-isomorphism $\gamma_0:M\rightarrow M'$ is then simply
$$\gamma_0(\oplus_j x_j\ten 1_{n_j})=\oplus_j 1_{n_j}\ten x_j^t,$$
where $t$ denotes transposition.

As $M_1=\cl{B}(L^2(M,\tau))$, the extended trace $\tau_1$ is the unique normalised trace on $\cl{B}(L^2(M,\tau))$, and $L^2(M_1,\tau_1)\cong L^2(M,\tau)\ten L^2(M,\tau)$, with $M_1$ (and hence $M$) acting on the first tensor leg. The conjugation $J_{1}$ is then the tensor flip (plus complex conjugation), so we have 
$$M_2=J_{1}(M\ten 1)'J_{1}=J_{1}(M'\ten \cl{B}(L^2(M,\tau)))J_{1}=\cl{B}(L^2(M,\tau))\ten M'.$$


In the Proposition below, we view $M_2=\cl{B}(L^2(M,\tau))\ten M'$ as a tripartite system, consisting of the (anti-)isomorphic commuting subalgebras $M\ten 1$, $M'\ten 1$ and $1\ten M'$. Alice has access to the first two, and Bob has access to the third. Here, the pertinent trace is $\tau_2:=\tau_1\ten\tau'$, where $\tau'(y)=\tau(Jy^*J)$, $y\in M'$.

\begin{proposition}\label{p:directsum} Let $M$ be a finite-dimensional $C^*$-algebra, and $M_1$ and $M_2$ be the result of the iterated basic construction of the inclusion $\C\subseteq M$ with respect to the Markov trace $\tau$. Then there exists a tight, minimal teleportation scheme $(\om,\{T_i\}_{i=1}^d,\{F_i\}_{i=1}^d)$ for $M$ relative to
$$\cl{B}(L^2(M,\tau))\ten 1, \ 1\ten M'\subseteq (M_2,\tau_2).$$
\end{proposition}

\begin{proof} As above, we take $M\cong\bigoplus_{j=1}^m M_{n_j}(\C)$ so that $M$ is represented on $L^2(M,\tau)$ as $M=\bigoplus_{j=1}^m M_{n_j}(\C)\ten 1_{n_j}$. 

Let $W_{j} : \Z_{n_j} \times \Z_{n_j} \rightarrow \mathcal{U}(M_{n_j}(\C))$ denote the Weyl representation of $ \Z_{n_j} \times \Z_{n_j}$, and $|\psi_{n_j}\ra$ denote the maximally entangled state in $\C^{n_j} \ten \C^{n_j}$ so that $|\psi_{n_j}\ra\la\psi_{n_j}|$ is the Jones projection for the inclusion $\C\subseteq M_{n_j}(\C)$ (cf. Example \ref{ex:1}). Define $F_{j,z_j}\in M\vee M'\subseteq M_1$ by 
$$F_{j,z_j} = 0_{n_1} \oplus ... \oplus 0_{n_{j-1}} \oplus(W_j(z_{j})^* \ten 1_{n_j})\ket{\psi_{n_j}}\bra{\psi_{n_j}}(W_j(z_{j}) \ten 1_{n_j})\oplus 0_{n_{j+1}} \oplus ... \oplus 0_{n_m},$$
for $j = 1,...n$ and $z_j \in \Z_{n_j}^2$. It follows that, $\sum_{z_j\in\Z_{n_j}^2}F_{j, z_j}=1_{n_j} \ten 1_{n_j}$ for each $j=1,...,n$. Hence, $\{ F_{j, z_j} \ | \ j = 1,...,n, z_{j} \in \Z_{n_j}^2 \}$ is a POVM in $M_1$ with cardinality $d:=\dim M$. Put 
$$\tilde{W}_{j}(z_j) = 1_{n_1} \oplus ... \oplus 1_{n_{j-1}} \oplus W_{j}(z_j) \oplus 1_{n_{j+1}} \oplus ... \oplus 1_{n_m},$$
which is a unitary in $M$, and define the automorphism $ T_{j, z_j} : M_2 \rightarrow M_2$ by 
$$ T_{j, z_j}(y) = \Gamma(\tilde{W}_j(z_j)) y \Gamma(\tilde{W}_j(z_j)^*), \ \ \ y\in M_2.$$
Clearly, each $ T_{j, z_j}$ leaves $M_1'\cap M_2=\Gamma(M)$ invariant, and is an $A$-bimodule map, where $A=\mc{B}(L^2(M,\tau))\ten 1$.

Finally, let $\om:=(\dim M)e_M$. Then $\om\in M'\cap M_2=M'\ten M'$. Noting that $\dim M=[M:\C]$, the Markov property (\ref{e:Markov}) implies that
$$\tau_2(\om)=\dim M(\tau_1\ten\tau')(e_M)=(\dim M)\tau_1(E_{M_1}(e_M))=\tau_1(1)=1,$$
so $\om$ is a $\tau_2$-density. We show that
\begin{equation*}
x=E_{M} \bigg( \sum_{j=1}^m \sum_{z_j \in \Z_{n_j}^2} F_{j,z_j} T_{j,z_j}(\Gamma(x))\om \bigg), \ \ \ x\in M,
\end{equation*}
where $E_M \equiv E_{M_1 \rightarrow M} \circ E_{M_2 \rightarrow M_1}$ is the unique trace preserving conditional  expectation from $M_2$ to $M$, and $\Gamma$ is the canonical shift from $M$ to $M_1' \cap M_2$. This will establish the Proposition, as tightness and minimality are clear by construction.

First consider the $ T_{j,z_j}(\Gamma(x))e_M$ term in the above equation (pulling the constant off $\om$).
\begin{align*}
 T_{j,z_j}(\Gamma(x))e_M= \Gamma\bigg( \tilde{W}_j(z_j) x \tilde{W}_j(z_j)^* \bigg )e_{M}
= \gamma_0\bigg( \tilde{W}_j(z_j) x \tilde{W}_j(z_j)^* \bigg )e_{M},
\end{align*}
the last equality following from Lemma \ref{l:me}. Write $x=\oplus_j (x_j\ten 1_{n_j})$. Then multiplying the above by $F_{j,z_j}$, we see that

\begin{align*}&F_{j,z_j}\gamma_0\bigg( \tilde{W}_j(z_j) x \tilde{W}_j(z_j)^* \bigg )e_M\\
&=F_{j,z_j}\gamma_0\bigg((x_1\ten 1_{n_1})\oplus\cdots\oplus (W_j(z_j)x_jW_j(z_j)^*\ten 1_{n_j})\oplus\cdots\oplus (x_m\ten 1_{n_m})\bigg)e_M\\
&=\bigg( (W_j(z_{j})^* \ten 1_{n_j})\ket{\psi_{n_j}}\bra{\psi_{n_j}}(W_j(z_{j}) \ten 1_{n_j}) \bigg) \bigg(1_{n_j} \ten \overline{W_j(z_j)}x_j^tW_j(z_j)^t\bigg)e_M\\
&=\bigg( (W_j(z_{j})^* \ten 1_{n_j})\ket{\psi_{n_j}}\bra{\psi_{n_j}}(1_{n_j} \ten \overline{W_j(z_j)}x_j^tW_j(z_j)^t)(W_j(z_{j}) \ten 1_{n_j}) \bigg)e_M\\
&=(W_j(z_{j})^* \ten 1_{n_j})\ket{\psi_{n_j}}\bra{\psi_{n_j}}((\overline{W_j(z_j)}x_j^tW_j(z_j)^t)^tW_j(z_{j}) \ten 1_{n_j})e_M \ \ \ \ \ \ \textnormal{(Lemma \ref{l:me} with $\ket{\psi_{n_j}}\bra{\psi_{n_j}}$)}\\
&=(W_j(z_{j})^* \ten 1_{n_j})\ket{\psi_{n_j}}\bra{\psi_{n_j}}(W_j(z_j)x_j \ten 1_{n_j})e_M\\
&=F_{j,z_j}xe_M
\end{align*}
Hence,
\begin{align*} E_{M} \bigg( \sum_{j=1}^m \sum_{z_j \in \Z_{n_j}^2} F_{j,z_j} T_{j,z_j}(\Gamma(x)))\om \bigg)&=\dim ME_M\bigg(\sum_{j=1}^m \sum_{z_j \in \Z_{n_j}^2}F_{j,z_j}xe_M\bigg)\\
&=x\dim ME_M(e_M)\\
&=x,
\end{align*}
where the last equality uses the Markov property 
\[E_M(e_M)=E_M(E_{M_1}(e_M))=\dim M^{-1}=[M:\C]^{-1}.\qedhere\]
\end{proof}

\begin{remark} Proposition \ref{p:directsum} shows that the ``direct sum'' of the standard protocols allows one to teleport any finite-dimensional algebra $M$. Moreover, the scheme is tight in the sense that the number of classical messages sent by Alice to Bob coincides with the dimension of the algebra to be teleported. However, one drawback of this approach is that Alice's measurement results are correlated with the location of the state within $M=\bigoplus_j M_{n_j}(\C)$. Indeed, if $\rho\in M$ is a density living entirely in one summand, say $M_{n_j}(\C)$, then the probability of Alice measuring outcome $(i,z_i)$, $i\neq j$, is $\tau(F_{i,z_i}\rho\om)=0$.

It could be desirable that Alice's measurement result, i.e., the classical information sent to Bob, contains no information about the state to be teleported. This is precisely the unbiased condition, and one can imagine this feature to be important in potential applications of the teleportation schemes introduced here; for instance, in the context of quantum privacy applications such as those considered in \cite{amtdw,bhs,brs,br03,ckpp,cgs,cgl,fiedler2017jones,jklp1,kribs2021approximate,levick2017quantum}. If $\C\subseteq M$ admits an orthonormal basis (with respect to the Markov trace) consisting of unitaries, then we can obtain an unbiased scheme. In fact, the same is true for more general inclusions $N\subseteq M$.
\end{remark}

\begin{theorem}\label{t:tele} Let $N\subseteq M$ be a strongly Markov inclusion of tracial von Neumann algebras. Suppose there is an orthonormal Pimsner--Popa basis $\{u_i\}_{i=1}^d$ for $M$ over $N$ inside the normaliser $\mathcal{N}_M(N)$, that $N'\cap M$ is finite-dimensional, and that $N'\cap M\subseteq L\subseteq N'$ for some injective factor $L$. Then there exists an unbiased teleportation scheme $(\om,\{F_i\}_{i=1}^d,\{ T_i\}_{i=1}^d)$ for $N'\cap M$, with respect to $M_1, M_1'\cap M_2\subseteq M_1\vee (M_1'\cap M_2)$.
\end{theorem}

In the proof, the following Lemma will be used to achieve the locality condition between Alice's and Bob's operations in Definition \ref{d:tele}. It is an application of the periodicity of the tower of the basic construction (see, e.g., \cite[Proposition 1.5]{pp}, or \cite[Proposition 4.3.7]{js}). We include details for the convenience of the reader (since we could not find a proof beyond the II$_1$-factor setting).

\begin{lemma}\label{l:LOCC} Let $N\subseteq M$ be a strongly Markov inclusion of tracial von Neumann algebras. Suppose there is an orthonormal Pimsner--Popa basis $\{u_i\}_{i=1}^d$ for $M$ over $N$ in the normaliser $\cl N_M(N)$. Then there exist unitaries $\{v_i\}_{i=1}^d\subseteq M_2$ satisfying
$$v_i\Gamma(x)v_i^*=\Gamma(u_ixu_i^*), \ \ \ i=1,...,d, \ x\in N'\cap M.$$
If, in addition, $N'\cap M$ is finite-dimensional and $N'\cap M\subseteq L\subseteq N'$ for some injective factor $L$, then $\Ad(v_i)|_{M_1'\cap M_2}$ extends to a $M_1$-bimodule $*$-automorphism of $M_1\vee (M_1'\cap M_2)$.
\end{lemma}

\begin{proof} Define $\vphi:M_d(M)\rightarrow M_2$ by 
$$\phi([x_{i,j}])=[M:N]\sum_{i,j=1}^d u_i^*e_Nx_{i,j}e_Me_Nu_j, \ \ \ x\in M_d(M).$$ 
Using orthonormality of $\{u_i\}$ together with the relations $e_M(\cdot)e_M=E_M(\cdot)e_M$ and $e_N(\cdot)e_N=E_N(\cdot)e_N$ on $M_1$ and $M$, respectively, it follows that $\phi$ is multiplicative:
\begin{align*}\phi([x_{i,j}])\phi([y_{i,j}])&=[M:N]^2\sum_{i,j,k,l=1}^du_i^*e_Nx_{i,j}e_M(e_Nu_ju_k^*e_Ny_{k,l})e_Me_Nu_l\\
&=[M:N]^2\sum_{i,j,k,l=1}^du_i^*e_Nx_{i,j}E_M(e_Nu_ju_k^*e_Ny_{k,l})e_Me_Nu_l\\
&=[M:N]^2\sum_{i,j,k,l=1}^du_i^*e_Nx_{i,j}E_M(E_N(u_ju_k^*)e_N)y_{k,l}e_Me_Nu_l\\
&=[M:N]^2\sum_{i,j,l=1}^du_i^*e_Nx_{i,j}E_M(e_N)y_{j,l}e_Me_Nu_l\\
&=[M:N]\sum_{i,j,l=1}^du_i^*e_Nx_{i,j}y_{j,l}e_Me_Nu_l\\
&=\phi([x_{i,j}][y_{k,l}]).\\
\end{align*}
Also, as $e_M\in M'$, 
$$\phi([x_{i,j}])^*=[M:N]\sum_{i,j=1}^d u_j^*e_Ne_Mx_{i,j}^*e_Nu_i=[M:N]\sum_{i,j=1}^d u_i^*e_Nx_{j,i}^*e_Me_Nu_j=\phi([x_{i,j}]^*).$$
Also, the relation $e_Ne_Me_N=[M:N]^{-1}e_N$ (see (\ref{e:TL})) together with the basis property implies that $\phi$ is unital. (One can also show that $\vphi$ is bijective, and hence a $*$-isomorphism. This fact is well known for \textit{any} finite-index inclusion of II$_1$-factors, see, e.g., \cite[Proposition 4.3.7]{js}.)

For each $i=1,...,d$, define $v_i:=\vphi(\mathrm{diag}(u_i,u_i,...,u_i))\in M_2$. Then $$v_i=[M:N]\sum_{j=1}^d u_j^*e_Nu_ie_Me_Nu_j$$
is unitary as the image of a unitary in $M_d(M)$ under a unital $*$-homomorphism. 

Let $x\in N'\cap M$. By Lemma \ref{l:me2} we have $e_N\Gamma(x)e_M=e_N xe_M$. Also, by orthonormality and the fact that $\Gamma(x)\in M_1'\cap M_2$, for each $j,k=1,...d$ we have
$$e_Nu_j\Gamma(x)u_k^*e_N=E_N(u_ju_k^*)e_N\Gamma(x)=\delta_{j,k}e_N\Gamma(x).$$
Thus,
\begin{align*}v_i\Gamma(x)v_i^*&=[M:N]^2\sum_{j,k=1}^du_j^*e_Nu_ie_M(e_Nu_j\Gamma(x)u_k^*e_N)e_Mu_i^*e_Nu_k\\
&=[M:N]^2\sum_{j=1}^du_j^*e_Nu_ie_M(e_N\Gamma(x)e_M)u_i^*e_Nu_j\\
&=[M:N]^2\sum_{j=1}^du_j^*e_Nu_ie_M(e_Nxe_M)u_i^*e_Nu_j\\
&=[M:N]^2\sum_{j=1}^du_j^*e_Nu_i(e_Me_Ne_M)xu_i^*e_Nu_j\\
&=[M:N]\sum_{j=1}^du_j^*e_N(u_ie_M)xu_i^*e_Nu_j\\
&=[M:N]\sum_{j=1}^du_j^*e_Ne_M(u_ixu_i^*)e_Nu_j\\
&=[M:N]\sum_{j=1}^du_j^*e_Ne_M\Gamma(u_ixu_i^*)e_Nu_j\\
&=[M:N]\bigg(\sum_{j=1}^du_j^*e_Ne_Me_Nu_j\bigg)\Gamma(u_ixu_i^*)\\
&=\bigg(\sum_{j=1}^du_j^*e_Nu_j\bigg)\Gamma(u_ixu_i^*)\\
&=\Gamma(u_ixu_i^*).
\end{align*}
Now, suppose, in addition, that $N'\cap M$ is finite-dimensional and $N'\cap M\subseteq L\subseteq N'$ for some injective factor $L$. Then $L_0:=J_1J_0(L)J_0J_1\subseteq M_1'$ is an injective factor such that $M_1'\cap M_2\subseteq L_0\subseteq M_1'$. By injectivity of $L_0$, multiplication induces a $*$-isomorphism $m:L_0'\ten_{\min} L_0\cong C^*(L_0',L_0)$, the $C^*$-subalgebra of $\mc{B}(L^2(M_1,\tau_1))$ generated by $L_0$ and $L_0'$ \cite[Corollary 4.6]{el}. By injectivity of the minimal tensor product, $M_1\ten_{\min} (M_1'\cap M_2)\subset L_0'\ten_{\min} L_0$. It follows that the restriction of $m$ induces a $*$-isomorphism 
$$m:M_1\ten_{\min} (M_1'\cap M_2)\cong C^*(M_1,M_1'\cap M_2).$$
Note also that by finite-dimensionality of $M_1'\cap M_2=\Gamma(N'\cap M)$, 
$$M_1\ten_{\min} (M_1'\cap M_2)=M_1\oten (M_1'\cap M_2),$$
where $\oten$ is the von Neumann tensor product. Since $M_1$ is weakly closed and commutes with the finite-dimensional $C^*$-algebra $M_1'\cap M_2$, by considering a system of matrix units for the latter algebra we can see that $C^*(M_1,M_1'\cap M_2)$ is weakly closed, and hence coincides with $M_1\vee (M_1'\cap M_2)$. The desired extension of $\Ad(v_i)$ is then the composition
$$M_1\vee (M_1'\cap M_2)\cong M_1\oten (M_1'\cap M_2)\xrightarrow{(\id\ten\Ad(v_i))}M_1\oten (M_1'\cap M_2)\cong M_1\vee (M_1'\cap M_2).$$
Note that the extension is necessarily an $M_1$-bimodule $*$-automorphism.
\end{proof}


\begin{proof}[Proof of Theorem \ref{t:tele}] Let $\om:=[M:N]e_M$. Then $\om$ is a $\tau$-density in $M'\cap M_2$ by the Markov property. For each $i=1,...,d$, let $F_i=u_i^*e_Nu_i$. Then $\{F_i\}_{i=1}^d$ is a PVM in $N'\cap M_1$. Finally, let $v_i:=\vphi(\mathrm{diag}(u_i,u_i,...,u_i))\in M_2$ be the unitary associated to $u_i$ from Lemma \ref{l:LOCC}, and define $ T_i$ to be the $M_1$-bimodule extension of $\mathrm{Ad}(v_i)|_{M_1'\cap M_2}$ to $M_1\vee (M_1'\cap M_2)$. It follows that $\sum_{i=1}^d \Ad(F_i)\circ T_i$ is a one-way (right) LOCC operation relative to $M_1,M_1'\cap M_2\subseteq M_1\vee (M_1'\cap M_2)$.

Fix $x\in N'\cap M$. We will show that
$$x=\sum_{i=1}^dE_{N'\cap M}(F_i T_i(\Gamma(x))\om),$$
so that $(\om,\{F_i\}_{i=1}^d,\{ T_i\}_{i=1}^d)$ forms a teleportation scheme. By Lemma \ref{l:LOCC}
$$T_i(\Gamma(x))=v_i\Gamma(x)v_i^*=\Gamma(u_ixu_i^*).$$
Applying Lemma \ref{l:me} twice, we see that
\begin{align*}
E_{N'\cap M}(F_i T_i(\Gamma(x))\om)&=[M:N]E_{N'\cap M}(F_i T_i(\Gamma(x))e_M)\\
&=[M:N]E_{N'\cap M}(F_i\Gamma(u_ixu_i^*)e_M)\\
&=[M:N]E_{N'\cap M}(F_i\gamma_0(u_ixu_i^*)e_M)\\
&=[M:N]E_{N'\cap M}(u_i^*e_Nu_i\gamma_0(u_ixu_i^*)e_M)\\
&=[M:N]E_{N'\cap M}(u_i^*e_N\gamma_0(u_ixu_i^*)u_ie_M)\\
&=[M:N]E_{N'\cap M}(u_i^*e_N(u_ixu_i^*)u_ie_M)\\
&=[M:N]E_{N'\cap M}(u_i^*e_Nu_ixe_M)\\
&=[M:N]E_{N'\cap M}(F_ie_M)x.\\
\end{align*}
Summing over $i$, and applying the Markov property 
$$E_{N'\cap M}(e_M)=E_{N'\cap M}(E_{M_1}(e_M))=[M_1:M]^{-1} 1=[M:N]^{-1}1$$
gives the identity.

It remains to show the unbiased property, that is, $\tau(F_i\rho\om)=[M:N]^{-1}$ for all $i$ and $\tau$-densities $\rho\in N'\cap M$. For any such $\rho$, $F_i\rho\in M_1$, so that $\tau(F_i\rho e_M)=[M:N]^{-1}\tau(F_i\rho)$. Thus,
\begin{align*}\tau(F_i\rho\om)&=[M:N]\tau(F_i\rho e_M)=\tau(u_i^*e_Nu_i\rho)\\
&=\tau(e_N(u_i\rho u_i^*))=[M:N]^{-1}\tau(u_i\rho u_i^*)\\
&=\frac{1}{[M:N]}.
\end{align*}
\end{proof}

\begin{remark} The unbiased scheme in Theorem \ref{t:tele} is not necessarily tight. For example, tightness would imply $[M:N]=\dim(N'\cap M)$, which, for a finite-index inclusion $N\subseteq M$ of II$_1$-factors, means that the inclusion has depth 1 \cite[Theorem 4.6.3(vii)]{dhj}.
\end{remark}

The main assumption in Theorem \ref{t:tele} was the existence of a Pimsner--Popa basis $\{u_i\}$ for $M$ over $N$ inside the normaliser $\mathcal{N}_M(N)$. In this case, the completeness relation $\sum_i u_i^*e_N u_i=1$  corresponds to the decomposition 
\begin{equation}\label{e:decomp}L^2(M,\tau)=\bigoplus_i u_i^*L^2(N,\tau)\end{equation}
of $L^2(M,\tau)$ into maximally entangled subspaces with respect to the commuting subsystems $N'\cap M$ and $M'\cap M_1$ (see Section \ref{s:ent}). When $N=\C\subseteq M_n(\C)=M$, the normaliser assumption holds trivially, and the decomposition (\ref{e:decomp}) simply corresponds to an orthonormal basis of maximally entangled vectors of $L^2(M_n(\C),\tau_n)=\C^n\ten\C^n$, i.e., Alice's local measurement in teleportation. Moreover, the normalisation property implies that the decomposition (\ref{e:decomp}) is one of $N$-bimodules, suggesting a potential connection with categorical approaches to teleportation \cite{ac,lwj}, which we leave for future investigations (see Outlook section).

\begin{example}\label{ex:PPbases} Recall that the hypotheses of Theorem~\ref{t:tele} require the existence of an orthonormal Pimsner-Popa basis in the normaliser. This holds in the following cases:
${}$
\begin{enumerate}
\item Any inclusion $M\subseteq G\ltimes M$, where $G$ is a finite group with a trace-preserving action on a tracial von Neumann algebra $M$ (see Example \ref{ex:3}).
\item Any finite-index regular inclusion $N\subseteq M$ of II$_1$-factors \cite{bg1,bg2,ckp}. This fact was shown by Bakshi-Gupta whenever $N'\cap M$ is either simple or commutative \cite{bg1,bg2} (follow the proof of \cite[Theorem 3.21]{bg2}), and was recently extended to any finite-index regular inclusion (of II$_1$-factors) by Crann-Kribs-Pereira \cite{ckp}.
\item Any inclusion $N\subseteq M_n(\C)$ with $N=\bigoplus_{j=1}^k M_l(\C)$ homogeneous. (This may be seen as a special case of (1).) Homogeneous algebras of this type can be used to model hybrid quantum codes \cite{bkk1,cao2021higher,grassl2017codes,nemec2018hybrid,nemec2021infinite}, so our scheme from Theorem \ref{t:tele} could theoretically be used to teleport hybrid quantum codes in an unbiased manner.
\end{enumerate}
\end{example}

We end this section with a partial converse of Example \ref{ex:PPbases} (3).

\begin{proposition} A multiplicity-free inclusion $N\subseteq M_n(\C)$ admits a Pimsner-Popa basis in $\mc{N}_{M_n(\C)}(N)$ if and only if $N$ is homogeneous.
\end{proposition}

\begin{proof} If $N\subseteq M_n(\C)$ is multiplicity free and $N$ is homogeneous, then $n=kl$ and $N=\bigoplus_{j=1}^k M_l(\C)\cong\ell^\infty_k\ten M_l(\C)$. The existence of the required basis follows from Example \ref{ex:PPbases} (3).

Suppose $N = \bigoplus_{j=1}^k M_{n_j}(\C) \subseteq M_n(\C)$ is a multiplicity-free inclusion. Let $z_1,...,z_k$ be the minimal central projections of $N$, indexed such that $z_jN \cong M_{n_j}(\C)$. It follows that $E_N=\sum_{j=1}^k z_j(\cdot)z_j$ and that $e_N=\sum_{j=1}^k z_k\ten z_k\in \mc{B}(\C^n\ten\C^n)$.

Let $\{ \lm_i \}_{i=1}^d$ be any Pimsner-Popa basis for $M_n(\C)$ over $N$.  Then
\begin{align*}
1_n\ten 1_n &= \sum_{i=1}^d (\lm_i^* \ten 1_n)e_N (\lm_i \ten 1_n) \\ 
&= \sum_{i=1}^d (\lm_i^* \ten 1_n) \bigg(\sum_{j=1}^k z_j \ten z_{j}\bigg) (\lm_i \ten 1_n) \\
&= \sum_{j=1}^k \bigg(\sum_{i=1}^d \lm_i^*z_j\lm_i\bigg) \ten z_{j}.
\end{align*}
If each $\lambda_i$ is unitary, then we can multiply by $1_n\otimes z_j$ and take unnormalised traces to obtain
\[ n n_j=\tr(1_n\otimes z_j)=\tr\bigg(\big(\sum_{i=1}^d \lm_i^*z_j\lm_i\big)\otimes z_j\bigg) = d n_j^2,\]
so $n_j=n/d$ for every $j$, hence $N$ is homogeneous.
\end{proof}

\section{Rigidity of Teleportation for Finite-Dimensional Inclusions}

Werner \cite{w} established a one-to-one correspondence between tight teleportation schemes for the tripartite system $M_n(\C)\ten M_n(\C)\ten M_n(\C)$ and orthonormal bases of unitaries of $M_n(\C)$ (a.k.a, unitary error bases). In this section we generalize Werner's result to inclusions of the form $N\subseteq M_n(\C)$.

\begin{theorem}\label{t:Werner} Let $N\subseteq (M_n(\C),\tau)$ be an inclusion such that $\tau|_{N'}$ is the Markov trace for $\C\subseteq N'$. Put $A_0:=N'\ten 1\ten 1$, $A_1:=1\ten N'\ten 1$ and $B=1\ten 1\ten N'$. If $(\om,\{F_i\}_{i=1}^d,\{T_i\}_{i=1}^d)$ is a tight, minimal, faithful teleportation scheme for $A_0$, with respect to
 $$M_n(\C)\ten M_n(\C)\ten 1, \ 1\ten 1\ten N'\subseteq M_n(\C)\ten M_n(\C)\ten N',$$
then there exist (1) an orthonormal basis $\{u_i\}_{i=1}^d$ of $M_n(\C)$ over $N$ in $\mathcal{N}_{M_n(\C)}(N)$, (2) a unitary $u\in\mathcal{N}_{M_n(\C)}(N)$, (3) a positive invertible operator $z\in \mathcal{Z}(N)$ such that
\begin{itemize}
\item $\om=[M_n(\C):N](1\ten z^{1/2}u)e_N(1\ten u^*z^{1/2})$;
\item $F_i=(u_i^*u\ten 1)e_N(u^*u_i\ten 1)$, $i=1,...,d$;
\item $T_i(x)=u_ixu_i^*$, $x\in N'$.
\end{itemize}
\end{theorem}

The proof requires several preparations. We follow the same general strategy as in \cite{w}, although our setting does not allow us to work at the level of individual Kraus operators. Instead, we mostly argue at the level of CP maps, which, incidentally, allows us to circumnavigate (the analogue of) \cite[Proposition 3]{w}, albeit we begin with the a priori stronger assumption of faithfulness.

The first Lemma is a simple (known) observation.

\begin{lemma}\label{l:normalise} Let $N\subseteq(M,\tau)$ be an inclusion of finite von Neumann algebras. Write $\cl N_M(N):=\{u\in \cl U(M)\mid u^*Nu=N\}$, and $\cl N_M(E_N):=\{u\in\mc{U}(M)\mid u^*E_N(x)u=E_N(u^*xu), \ x\in M\}$. Then we have the equality 
$$\cl N_M(N)=\cl N_M(E_N).$$
Moreover, any $u\in \mc{N}_M(N)$ satisfies
$$ue_Nu^*=\pi^r_\tau(u)e_N\pi^r_\tau(u^*).$$
\end{lemma}

\begin{proof} The inclusion $\cl N_M(E_N)\subseteq \cl N_M(N)$ follows easily from the definitions. Given $u\in\cl N_M(N)$, it follows that $\mathrm{Ad}(u)\circ E_N\circ\mathrm{Ad}(u^*)$ is a $\tau$-preserving conditional expectation from $M$ to $N$. By uniqueness, $\mathrm{Ad}(u)\circ E_N\circ\mathrm{Ad}(u^*)=E_N$, which is the desired normalisation property.

If $u\in \mc{N}_M(N)=\cl N_M(E_N)$, then for all $x\in M$,

\begin{align*}u e_Nu^*\Ltau(x)&=\Ltau(uE_N(u^*x))=\Ltau(uE_N(u^*x)u^*u)=\Ltau(E_N(xu^*)u)\\
&=\pi^r_\tau(u)\Ltau(E_N(xu^*))=\pi^r_\tau(u)e_N\Ltau(xu^*)\\
&=\pi^r_\tau(u)e_N\pi^r_\tau(u^*)\Ltau(x).
\end{align*}
\end{proof}

The next two Lemmas generalise \cite[Lemma 7]{w} and \cite[Lemma 2]{w}, respectively.

\begin{lemma}\label{l:Lemma7}
Let $M$ be a finite-dimensional $C^*$-algebra and let $\{ T_i\}_{i=1}^d$ be a family of CP maps $ T_i : M \rightarrow M$ such that $\sum_{i=1}^d  T_i = \id_M$. If $M = \bigoplus_{j=1}^m z_jM$, where $z_1,...,z_m,$ are the minimal central projections of $M$, then for each $i,j$ there exists $\mu_i^j \in[0,\infty)$ such that 
$$T_i|_{z_jM} = \mu_i^j\id_{z_jM}.$$
\end{lemma}

\begin{proof}
Let $p$ a minimal projection of $M$. Then, 
$$0 \leq  T_i(p) \leq \sum_i T_i(p) = p,$$
$ T_i(p) = p T_i(p)p \in pMp = \C p$ by minimality. It follows that $ T_i(z_jM) \subseteq z_jM$, for each $i$ and $j$. Then for each $j$, $ T_i|_{z_jM} = z_jM \rightarrow z_jM$ is CP and $\sum_i  T_i|_{z_jM} = \id_{z_jM}$. By \cite[Lemma 7]{w}, each $ T_i|_{z_jM} = \mu_i^j \id_{z_jM}$.
\end{proof}

\begin{remark} The coefficients $\mu_i^j$ depend on the summand $j$, in general, and can be zero even if $T_i\neq 0$. For instance, consider $T_1,T_2:\ell^\infty_2\rightarrow\ell^\infty_2$ given by
$$T_1(x,y)=(2^{-1}x,0), \ \ \ T_2(x,y)=(2^{-1}x,y), \ \ \ (x,y)\in\ell^\infty_2.$$
Then $T_1$ and $T_2$ are CP and $T_1+T_2=\id_{\ell^\infty_2}$. The associated coefficients are $\mu_1^1=\frac{1}{2}$, $\mu_1^2=0$, $\mu_2^1=\frac{1}{2}$, $\mu_2^2=1$.
\end{remark}

\begin{lemma}\label{l:ONB}
Let $N \subseteq (M,\tau)$ be a strongly Markov inclusion of finite-dimensional von Neumann algebras. A basis $\{\lambda_i\}_{i=1}^d$ for $M$ over $N$ is orthonormal (i.e., $E_N(\lm_i\lm_j^*)=\delta_{i,j}1$) if and only if,
$$d = \frac{\dim M}{\dim N}.$$
\end{lemma}

\begin{proof}
Assume $d = \dim M/\dim N$, and let $\{ e_n \ | \ n =1,...,\dim N\}$ be an orthonormal basis for $L^2(N,\tau)$. Then $e_N = \sum_{n=1}^{\dim N} |e_n\ra\la e_n|$, so that 
$$1_{L^2(M,\tau)} 
= \sum_{i=1}^d \lambda_i^* e_N \lambda_i 
= \sum_{i=1}^d \sum_{n=1}^{\dim N} \lambda_i^* |e_n\ra\la e_n| \lambda_i.$$
Since $d\dim N = \dim M = \dim(L^2(M,\tau))$, by \cite[Lemma 2]{w}, the set $\{\lambda_i^*|e_n\ra \ | \ i =1,...,d; \ n=1,...,\dim N \},$ is an orthonormal basis of $L^2(M,\tau)$. In particular, the subspaces $\lambda_i^*L^2(N,\tau)$ are orthogonal, $\{\lambda_i^*e_N\lambda_i\}_{i=1}^d$ are mutually orthogonal projections and each $\lm_i^*$ acts isometrically on $L^2(N,\tau)$. Hence, for every $\xi,\eta\in L^2(M,\tau)$
\begin{align*}\la E_N(\lm_i\lm_j^*)e_N\xi,\eta\ra&=\la e_N\lm_i\lm_j^*e_N\xi,\eta\ra\\
&=\la \lm_j^*e_N\xi,\lm_i^*e_N\eta\ra\\
&=\delta_{i,j}\la \lm_i^*e_N\xi,\lm_i^*e_N\eta\ra\\
&=\delta_{i,j}\la e_N\xi,\eta\ra.\\
\end{align*}
Since the map $N \ni y \rightarrow ye_N \in \mc{B}(L^2(M,\tau))$ is injective (by faithfulness of $\tau$), $E_N(\lambda_i\lambda_j^*) = \delta_{ij}1$.

Conversely, suppose $\{ \lambda_i \}_{i=1}^d$ is orthonormal. Then, 
$$(\lambda_i^*e_N\lambda_i)(\lambda_j^*e_N\lambda_j)=\lambda_i E_N(\lm_i\lm_j^*)e_N\lm_j=\delta_{i,j}\lm_i^*e_N\lm_i,$$
so $\lambda_i^*e_N\lambda_i$ are mutually orthogonal projections, and $\lambda_i^*$ is isometric on $L^2(N,\tau):$
\begin{align*}
\langle \lambda_i^*e_N\xi, \lambda_i^*e_N\eta \rangle 
&= \langle e_N\lambda_i\lambda_i^*e_N\xi, \eta \rangle \\ 
&= \langle E_N(\lambda_i\lambda_i^*)e_N \xi, \eta \rangle \\ 
&= \langle e_N\xi, \eta \rangle \\
&= \langle e_N\xi, e_N\eta \rangle.
\end{align*} 
It follows that each projection $\lambda_i^*e_N\lambda_i$ has rank $\dim N$. The identity $1_{L^2(M,\tau)} = \sum_{i=1}^d \lambda_i^*e_N\lambda_i$ then implies that $\dim M = d\dim N$.
\end{proof}

The next Lemma will produce useful decompositions for the operations in the teleportation scheme of Theorem \ref{t:Werner}.
 
\begin{lemma}\label{l:Choi}
Let $N\subseteq M$ be a strongly Markov inclusion with basis $\{\lm_i\}_{i=1}^d$. Then for each positive element $x_1\in N'\cap M_1$, there exists $\{a_i\}_{i=1}^d$ in $M$ such that
$$x_1=\sum_{i=1}^d a_i^*e_Na_i \ \ \ \textnormal{and} \ \ \ \sum_{i=1}^da_i^*(N'\cap M)a_i\subseteq N'\cap M.$$
Moreover, the CP map $\sum_ia_i^*(\cdot)a_i$ on $N'\cap M$ is independent of the chosen basis $\{\lm_i\}_{i=1}^d$. If, in addition, $M$ admits a unitary basis $\{u_i\}_{i=1}^d$ over $N$ in $\mc{N}_M(N)$, then we can choose the $a_i$ so that $a_i^*(N'\cap M)a_i, a_i(N'\cap M)a_i^*\subseteq N'\cap M$. 
\end{lemma}

\begin{proof}
By positivity of $x_1$, we get  $x_1=\sqrt{x_1}\cdot\sqrt{x_1}=\sum_{i=1}^d\sqrt{x_1}\lm_i^*e_N\lm_i\sqrt{x_1}$.
Since $\sqrt{x_1}\lm_i^*\in M_1$, there exists a unique element $a_i^*\in M$ such that $a_i^*e_N=\sqrt{x_1}\lm_i^*e_N$ \cite[Lemma 4.3.1, Remark 4.3.2(a)]{js}, namely $a_i^*=[M:N]E_M(\sqrt{x_1}\lm_i^*e_N)$. The first equation follows. 

As for the inclusion, noting that $a_i^*=[M:N]E_M(a_i^*e_N)$, for $x\in N'\cap M$ we have
\begin{align*}
a_i^*xa_i&=[M:N]E_M(a_i^*e_N)xa_i=[M:N]E_M(a_i^*e_Nxa_i)=[M:N]E_M(a_i^*e_N\gamma_0(x)a_i) \ \ \ (\textnormal{Lemma \ref{l:me}})\\
&=[M:N]E_M(a_i^*e_Na_i\gamma_0(x)) \ \ \ (\gamma_0(x)\in M'\cap M_1),
\end{align*}
and $\sum_{i=1}^d a_i^*e_Na_i=x_1$, so $\sum_{i=1}^d a_i^*xa_i=[M:N]E_M(x_1\gamma_0(x))$ 
which belongs to $N'\cap M$ since $x_1,\gamma_0(x)\in N'$. If $\{\mu_j\}$ was another basis for $M$ over $N$, and $b_j^*=[M:N]E_M(\sqrt{x_1}\mu_j^*e_N)$, then decomposing $\lm_i$ relative to $\{\mu_j\}$, we see that
\begin{align*}a_i^*&=[M:N] E_M(\sqrt{x_1}\lm_i^*e_N)\\
&=[M:N] \sum_{j}E_M(\sqrt{x_1}\mu_j^*E_N(\mu_j\lm_i^*)e_N)\\
&=[M:N] \sum_{j}E_M(\sqrt{x_1}\mu_j^*e_N)E_N(\mu_j\lm_i^*)\\
&=\sum_j b_j^*E_N(\mu_j\lm_i^*).\\
\end{align*}
Hence, for every $x\in N'\cap M$, we have 
\begin{align*}\sum_i a_i^*xa_i&=\sum_i\sum_{j,k} b_j^*E_N(\mu_j\lm_i^*)xE_N(\lm_i\mu_k^*)b_k\\
&=\sum_i\sum_{j,k} b_j^*E_N(\mu_j\lm_i^*E_N(\lm_i\mu_k^*))xb_k\\
&=\sum_{j,k} b_j^*E_N(\mu_j\mu_k^*)xb_k\\
&=[M:N]\sum_{j,k} E_M(\sqrt{x_1}\mu_j^*e_N)E_N(\mu_j\mu_k^*)xb_k\\
&=[M:N]\sum_{j,k} E_M(\sqrt{x_1}\mu_j^*E_N(\mu_j\mu_k^*)e_N)xb_k\\
&=[M:N]\sum_{k} E_M(\sqrt{x_1}\mu_k^*e_N)xb_k\\
&=\sum_{k} b_k^*xb_k.
\end{align*}

Now, suppose, in addition, that each $\lm_i=u_i\in \mc{N}_M(N)$ and $x_1\in (N'\cap M_1)^+$. Then defining the $a_i$ as above, for every $x\in N'\cap M$ and $y\in N$, we have

\begin{align*}y(a_i^*xa_i)&=[M:N]^2yE_M(\sqrt{x_1}u_i^*e_N)xE_M(e_Nu_i\sqrt{x_1})\\
&=[M:N]^2E_M(y\sqrt{x_1}u_i^*e_N)xE_M(e_Nu_i\sqrt{x_1}) \ \ \ \ \textnormal{(bimodule property of $E_M$)}\\
&=[M:N]^2E_M(\sqrt{x_1}(yu_i^*)e_N)xE_M(e_Nu_i\sqrt{x_1}) \ \ \ \ \textnormal{($\sqrt{x_1}\in N'$)}\\
&=[M:N]^2E_M(\sqrt{x_1}u_i^*(u_iyu_i^*)e_N)xE_M(e_Nu_i\sqrt{x_1})\\
&=[M:N]^2E_M(\sqrt{x_1}u_i^*e_N)xE_M(e_N(u_iyu_i^*)u_i\sqrt{x_1}) \ \ \ \ \textnormal{($u_iyu_i^*\in N$)}\\
&=[M:N]^2E_M(\sqrt{x_1}u_i^*e_N)xE_M(e_Nu_iy\sqrt{x_1})\\
&=[M:N]^2E_M(\sqrt{x_1}u_i^*e_N)xE_M(e_Nu_i\sqrt{x_1})y\\
&=(a_i^*xa_i)y.
\end{align*}

Similarly, 

\begin{align*}y(a_ixa_i^*)&=[M:N]^2yE_M(e_Nu_i\sqrt{x_1})xE_M(\sqrt{x_1}u_i^*e_N)\\
&=[M:N]^2E_M(e_N(yu_i)\sqrt{x_1})xE_M(\sqrt{x_1}u_i^*e_N)\\
&=[M:N]^2E_M(e_Nu_i(u_i^*yu_i)\sqrt{x_1})xE_M(\sqrt{x_1}u_i^*e_N)\\
&=[M:N]^2E_M(e_Nu_i\sqrt{x_1})xE_M(\sqrt{x_1}(u_i^*yu_i)u_i^*e_N) \ \ \ \ \textnormal{($u_i^*yu_i\in N$)}\\
&=[M:N]^2E_M(e_Nu_i\sqrt{x_1})xE_M(\sqrt{x_1}u_i^*ye_N)\\
&=(a_ixa_i^*)y.
\end{align*}
\end{proof}

Finally, we require the following equivalence with the trace assumption in Theorem \ref{t:Werner}, which will be used in conjunction with tightness to deduce orthonormality of the constructed basis.

\begin{lemma}\label{l:Markovtr} Let $N\subseteq (M_n(\C),\tau) $ be an inclusion such that $N=\bigoplus_{j=1}^JM_{n_j}(\C)\ten 1_{m_j}$. Then $\tau|_{N'}$ is the Markov trace for the inclusion $\C\subseteq N'$ if and only if 
$$\frac{n_j}{m_j}=\frac{n}{\mathrm{dim}N'}, \ \ \ j=1,...,J.$$
When this is the case, $(\id\ten\tau)e_N=(\tau\ten\id)e_N=[M_n(\C):N]^{-1}1$.

\end{lemma}

\begin{proof} Since $N'=\bigoplus_{j=1}^J1_{n_j}\ten M_{m_j}(\C)$, the trace vector associated to $\tau|_{N'}$ is $\frac{1}{n}(n_1,...,n_J)$. Since the Markov trace on $N'$ has trace vector $\frac{1}{\mathrm{dim}N'}(m_1,...,m_J)$, the first claim follows.

Let $M_1=\la M_n(\C),e_N\ra$ be the result of the basic construction. Then (abusing notation) $M_1=JN'J=M_n(\C)\ten N'$ and the canonical conditional expectation $E_{M_n(\C)}:M_1=M_n(\C)\ten N'\rightarrow M_n(\C)$ is simply $(\id\ten\tau_{N'})$, where $\tau_{N'}$ is the Markov trace for the inclusion $\C\subseteq N'$. Thus, assuming $\tau|_{N'}=\tau_{N'}$, the Markov property implies that
$$(\id\ten\tau)e_N=E_{M_n(\C)}(e_N)=[M_n(\C):N]^{-1}1.$$
As $e_N=Je_NJ$, we also have
$$(\tau\ten\id)e_N=(\tau\ten\id)Je_NJ=\overline{(\id\ten\tau)e_N}=[M_n(\C):N]^{-1}1.$$
\end{proof}

We are now in position to prove the main result of the section.

\begin{proof}[Proof of Theorem \ref{t:Werner}] Throughout the proof we let $M:=M_n(\C)$. Transposition/complex conjugation on $M_n(\C)$ will be taken relative to a block-diagonalising basis for $N'$.


Suppose $(\om,\{F_i\}_{i=1}^d,\{T_i\}_{i=1}^d)$ is a tight, minimal faithful teleportation scheme for $N'$, with respect to the bipartite system
$$M_n(\C)\ten M_n(\C)\ten 1, 1\ten 1\ten N'\subseteq M_n(\C)\ten M_n(\C)\ten N',$$
where we take $A_0=N'\ten 1\ten 1$ and $A_1=1\ten N'\ten 1$ inside $A=M_n(\C)\ten M_n(\C)\ten 1$, and $B=1\ten 1\ten N'$. Since $T_i$ is an $A$-bimodule map, it must be supported solely on Bob's algebra $B$. We therefore view each $T_i$ as a UCP map on $N'$.

Tightness means $d=\mathrm{dim}N'$, and minimality ensures that 
\begin{itemize}
\item $\{F_i\}_{i=1}^d\subseteq A_0\vee A_1= N'\ten N'\ten 1$, and
\item $\om\in A_1\vee B= 1\ten N'\ten N'$.
\end{itemize}
The teleportation identity then reads
\begin{equation}\label{e:tproof}x=\sum_{i=1}^d (E_{N'}\ten\tau\ten\tau)((F_i\ten 1)(1\ten1\ten T_i(x))(1\ten \om)), \ \ \ x\in N'.\end{equation}

Fixing a Pimsner-Popa basis $\{\lm_\alpha\}$ for $M$ over $N$ (which always exists for connected inclusions of finite-dimensional $C^*$-algebras \cite{b}), by Lemma \ref{l:Choi}, there exist families $\{w_\alpha\}$, $\{a_{i,\beta}\}$ in $M$ (indexed by the same set) satisfying
$$\om=\sum_{\alpha}(w_\alpha^*\ten 1)e_N(w_\alpha\ten 1), \ \ \ F_i=\sum_\beta (a_{i,\beta}^*\ten 1)e_N(a_{i,\beta}\ten 1), \ \ \ i=1,...,d.$$
Expanding the argument of $(E_{N'}\ten\tau\ten\tau)$ from equation (\ref{e:tproof}), we can rearrange and apply entanglement of $e_N$ (Lemma \ref{l:me}) as follows
\begin{align*}&(F_i\ten 1)(1\ten1\ten T_i(x))(1\ten \om)\\
&=\sum_{\alpha,\beta}((a_{i,\beta}^*\ten 1)e_N(a_{i,\beta}\ten 1)\ten 1)(1\ten 1\ten T_i(x))(1\ten ((w_\alpha^*\ten 1)e_N(w_\alpha\ten 1)))\\
&=\sum_{\alpha,\beta}((a_{i,\beta}^*\ten 1\ten 1)(e_N\ten 1)(a_{i,\beta}\ten w_\alpha^*\ten 1)(1\ten (1\ten T_i(x)e_N(w_\alpha\ten 1)))\\
&=\sum_{\alpha,\beta}((a_{i,\beta}^*\ten 1\ten 1)(e_N\ten 1)(a_{i,\beta}\ten w_\alpha^*T_i(x)^t\ten 1)(1\ten (e_N(w_\alpha\ten 1))).
\end{align*}
But then, as $\tau|_{N'}$ is the Markov trace for $\C\subseteq N'$, Lemma \ref{l:Markovtr} implies
\begin{align*}
&(E_{N'}\ten\tau\ten\tau)((F_i\ten 1)(1\ten1\ten T_i(x))(1\ten \om))\\
&=\sum_{\alpha,\beta}(E_{N'}\ten\tau\ten\tau)((a_{i,\beta}^*\ten 1\ten 1)(e_N\ten 1)(a_{i,\beta}\ten w_\alpha^*T_i(x)^t\ten 1)(1\ten (e_N(w_\alpha\ten 1)))\\
&=[M:N]^{-1}\sum_{\alpha,\beta}(E_{N'}\ten\tau)((a_{i,\beta}^*\ten 1)e_N(a_{i,\beta}\ten w_\alpha^*T_i(x)^tw_\alpha))\\
&=[M:N]^{-1}(E_{N'}\ten\tau)((\Psi_i\ten \id)(e_N)(1\ten \Phi(T_i(x)^t))),\\
\end{align*}
where $\Phi:=\sum_\alpha w_\alpha^*(\cdot)w_\alpha$ and $\Psi_i:=\sum_{\beta}a_{i,\beta}^*(\cdot)a_{i,\beta}$ are CP maps $N'\rightarrow N'$ (by Lemma \ref{l:Choi}). Applying entanglement of $e_N$ and the Markov property once again, we have
\begin{align*}&(E_{N'}\ten\tau\ten\tau)((F_i\ten 1)(1\ten1\ten T_i(x))(1\ten \om))\\
&=[M:N]^{-1}(E_{N'}\ten\tau)((\Psi_i\ten \id)(e_N(1\ten \Phi(T_i(x)^t))))\\
&=[M:N]^{-1}(E_{N'}\ten\tau)((\Psi_i\ten \id)(e_N(\Phi(T_i(x)^t)^t\ten 1)))\\
&=[M:N]^{-2}E_{N'}(\Psi_i(\Phi(T_i(x)^t)^t))\\
&=[M:N]^{-2}\Psi_i(\Phi(T_i(x)^t)^t)\\
&=[M:N]^{-2}\Psi_i(\Psi(T_i(x))),
\end{align*}
where $\Psi:=t\circ\Phi\circ t=\sum_{\alpha}w_{\alpha}^t(\cdot)\overline{w_\alpha}$ is a CP map $N'\rightarrow N'$ (as $N'$ is invariant under transposition). Thus, we have shown
$$\sum_{i=1}^d[M:N]^{-2}\Psi_i\circ\Psi\circ T_i =\id_{N'}.$$
Hence, by Lemma \ref{l:Lemma7}, if $N'=\bigoplus_{j=1}^Jz_jN'$, for minimal central projections $z_1,...,z_J$, then for each $i$, there exist $\mu_i^j,...,\mu_i^J\geq 0$ such that
$$[M:N]^{-2}\Psi_i\circ\Psi\circ T_i|_{z_jN'} =\mu_i^j\id_{z_jN'}.$$
Let $\sigma_i:=\sum_j \mu_i^jz_j\in\mc{Z}(N)$. Then for every $x\in N'$,
\begin{equation}\label{e:ucp}[M:N]^{-2}\Psi_i\circ\Psi\circ T_i(x)=\sum_{j=1}^J[M:N]^{-2}\Psi_i\circ\Psi\circ T_i(z_jx)=\sum_{j=1}^J\mu_i^jz_jx=\sigma_ix.
\end{equation}
In particular, $\sigma_i=[M:N]^{-2}\Psi_i\circ\Psi\circ T_i(1)=[M:N]^{-2}\Psi_i(\Psi(1))$.
Next, we show that $\sigma_i$ is invertible. Since the scheme is faithful, the element 
$$\rho_i:=(\id\ten\tau\ten\tau)((F_i\ten 1)(1\ten \om))\in N'$$
satisfies
$$\tau(\rho\rho_i)=(\tau\ten\tau\ten\tau)((F_i\ten 1)(\rho\ten 1\ten 1)(1\ten \om))>0$$
for all $\tau$-densities $\rho\in N'$. Thus, $\rho_i$ is invertible in $N'$. In fact, $\rho_i=\sigma_i$: expanding the operators $F_i$ and $\om$ once again using Lemma \ref{l:Choi} and performing similar manipulations using the entanglement and Markov property of $e_N$, we see that
\begin{align*} \rho_i&=\sum_{\alpha,\beta}(\id\ten\tau\ten\tau)(((a_{i,\beta}^*\ten 1)e_N(a_{i,\beta}\ten 1)\ten 1)(1\ten (w_\alpha^*\ten 1)e_N(w_\alpha\ten 1)))\\
&=(\id\ten\tau\ten\tau)(((\Psi_i\ten \id)(e_N)\ten 1)(1\ten (\Phi\ten\id)(e_N)))\\
&=[M:N]^{-1}(\id\ten\tau)((\Psi_i\ten \id)(e_N)(1\ten\Phi(1))) \ \ \ \ \textnormal{(trace out $3^{rd}$ leg)}\\
&=[M:N]^{-1}(\id\ten\tau)((\Psi_i\ten \id)(e_N(1\ten\Phi(1))))\\
&=[M:N]^{-1}(\id\ten\tau)((\Psi_i\ten \id)(e_N(\Phi(1)^t\ten 1))) \ \ \ \ \textnormal{($\Phi(1)\in N'$)}\\
&=[M:N]^{-2}\Psi_i(\Phi(1)^t)\\
&=[M:N]^{-2}\Psi_i(\Psi(1)).
\end{align*}
Hence, 
$$\sigma_i=[M:N]^{-2}\Psi_i\circ\Psi\circ T_i(1)=[M:N]^{-2}\Psi_i(\Psi(1))=\rho_i$$
is a positive invertible element of $\mc{Z}(N)$. Hence, $[M:N]^{-2}\sigma_i^{-1/2}\Psi_i(\Psi(\cdot))\sigma_i^{-1/2}$ is a UCP map $N'\rightarrow N'$ which is a left (hence two-sided) inverse to $T_i$ by equation (\ref{e:ucp}) (and finite-dimensionality of $N'$). Then $T_i$ is a unital complete order isomorphism of the unital $C^*$-algebra $N'$, so is necessarily a $*$-automorphism (see, e.g., \cite[Corollary 5.2.3]{er}). $T_i$ is therefore the restriction of a $*$-automorphism of $M$ to $N'$ (by the proof of \cite[Proposition 2.3.3]{dhj}, for instance), that is, $T_i(x)=u_ixu_i^*$ for some unitary $u_i\in \mc{N}_M(N')=\mc{N}_M(N)$. As shown above,
\begin{equation}\label{e:sigma}(E_{N'}\ten\tau\ten\tau)((F_i\ten 1)(1\ten1\ten T_i(x))(1\ten \om))=[M:N]^{-2}\Psi_i(\Psi(T_i(x)))=\sigma_ix, \ \ \ x\in N',\end{equation}
so we have
$$\sigma_i^{1/2}u_i^*xu_i\sigma_i^{1/2}=(E_{N'}\ten\tau\ten\tau)((F_i\ten 1)(1\ten1\ten x)(1\ten \om)), \ \ \ x\in N'.$$
Summing over $i$, and using the fact that $\{F_i\}_{i=1}^d$ is a POVM,
\begin{equation}\label{e:z}\sum_{i=1}^d\sigma_i^{1/2}u_i^*xu_i\sigma_i^{1/2}=(E_{N'}\ten\tau\ten\tau)((1\ten1\ten x)(1\ten \om))=\tau(x(\tau\ten\id)(\om))1, \ \ \ x\in N'.\end{equation}
Put $z:=(\tau\ten\id)(\om)\in N'$. The above relation together with the fact that $e_N\in N'\cap M_1=N'\ten N'$ implies
$$\sum_{i=1}^d(1\ten \sigma_i^{1/2}u_i^*)e_N(1\ten u_i\sigma_i^{1/2})=(\id\ten\tau)(e_N(1\ten z))\ten 1=(\id\ten\tau)(e_N(z^t\ten 1))\ten 1=[M:N]^{-1}z^t\ten 1.$$
Tracing out the right hand side, and using the fact that $u_i\sigma_i u_i^*\in N'$, 
$$[M:N]^{-1}z^t=\sum_{i=1}^d(\id\ten\tau)(e_N(1\ten u_i\sigma_i u_i^*))=\sum_{i=1}^d(\id\ten\tau)(e_N(\overline{u_i}\sigma_i^t u_i^t\ten 1)=[M:N]^{-1}\sum_{i=1}^d\overline{u_i}\sigma_i^t u_i^t,$$
implying $z=\sum_{i=1}^du_i\sigma_iu_i^*$ is a positive invertible element of $\mc{Z}(N)$. But then by (\ref{e:z})
$$\sum_{i=1}^d\sigma_i^{1/2}u_i^*z^{-1/2}xz^{-1/2}u_i\sigma_i^{1/2}=\tau(x)1, \ \ \ x\in N',$$
from which it follows that
$$\sum_{i=1}^d(\sigma_i^{1/2}u_i^*z^{-1/2}\ten 1)e_N(z^{-1/2}u_i\sigma_i^{1/2}\ten 1)=1\ten (\tau\ten\id)(e_N)=[M:N]^{-1}1\ten 1,$$
i.e., $\{\sqrt{[M:N]}z^{-1/2}u_i\sigma_i^{1/2}\}_{i=1}^d$ forms a basis for $M$ over $N$. However, $\tau|_{N'}$ is the Markov trace for $\C\subseteq N'$, so by Lemma \ref{l:Markovtr} we have
$$\frac{n_j}{m_j}=\frac{n}{\mathrm{dim}N'}, \ \ \ j=1,...,J,$$
where $N=\bigoplus_{j=1}^JM_{n_j}(\C)\ten 1_{m_j}$ is the decomposition induced from the inclusion $N\subseteq M_n(\C)$. Hence,
$$\dim N=\sum_{j=1}^Jn_j^2=\sum_{j=1}^J\frac{m_j^2n^2}{(\dim N')^2}=\frac{\dim N'\cdot n^2}{(\dim N')^2}=\frac{\dim M}{\dim N'}.$$
By tightness, we therefore have $d=\dim N'=\frac{\dim M}{\dim N}$, which entails the orthonormality of the basis $\{\sqrt{[M:N]}z^{-1/2}u_i\sigma_i^{1/2}\}_{i=1}^d$ by Lemma \ref{l:ONB}. As $u_i\sigma_iu_i^*\in\mc{Z}(N)$, we get
$$1=[M:N]E_N(z^{-1/2}u_i\sigma_iu_i^*z^{-1/2})=[M:N]z^{-1/2}u_i\sigma_iu_i^*z^{-1/2},$$
that is, $u_i\sigma_iu_i^*=[M:N]^{-1}z$. But then the basis elements are
$$\sqrt{[M:N]}z^{-1/2}u_i\sigma_i^{1/2}=\sqrt{[M:N]}z^{-1/2}(u_i\sigma_i^{1/2}u_i^*)u_i=u_i,$$
so $\{u_i\}_{i=1}^d$ is an orthonormal basis of $M$ over $N$ inside $\mc{N}_M(N)$. 

Next, consider the decomposition of $\om$ induced by $\{u_i\}_{i=1}^d$ as in Lemma \ref{l:Choi}, that is 
$$\om=\sum_i (w_i^*\ten 1) e_N(w_i\ten 1)$$
with each $w_i^*w_i, w_iw_i^*\in N'$, where $w_i^*=[M:N]E_M(\sqrt{\om}u_i^*e_N)$. Also by Lemma \ref{l:Choi}, the associated CP map $\Phi:N'\rightarrow N'$ (defined above) satisfies $\Phi=\sum_i w_i^*(\cdot)w_i$, so that $\Psi=t\circ\Phi\circ t=\sum_{i}w_i^t(\cdot)\overline{w_i}$. 

Let $\widetilde{\Psi}=[M:N]^{-1}\Psi\circ\Ad(z^{-1/2})$. From equation (\ref{e:sigma}) it follows that
\begin{equation}\label{e:Psi}\Psi_i(\widetilde{\Psi}(x))=[M:N]^{-1}\Psi_i(\Psi(z^{-1/2}xz^{-1/2}))=[M:N](\sigma_i^{1/2}u_i^*z^{-1/2})x(z^{-1/2}u_i\sigma_i^{1/2})=u_i^*xu_i\end{equation}
for every $x\in N'$. Moreover, $\widetilde{\Psi}$ is $\tau$-preserving on $N'$: 
\begin{align*}\tau(\widetilde{\Psi}(x))&=[M:N]^{-1}\tau(\Psi(z^{-1/2}xz^{-1/2}))\\
&=[M:N]^{-1}\sum_{i=1}^d\tau(w_i^tz^{-1/2}xz^{-1/2}\overline{w_i})\\
&=[M:N]^{-1}\sum_{i=1}^d\tau(z^{-1/2}xz^{-1/2}\overline{w_i}w_i^t)\\
&=\sum_{i=1}^d\tau(z^{-1/2}xz^{-1/2}(\tau\ten\id)(e_N(1\ten \overline{w_i}w_i^t)))\\
&=\sum_{i=1}^d\tau(z^{-1/2}xz^{-1/2}(\tau\ten\id)(e_N(w_iw_i^*\ten 1))) \ \ \ \ \textnormal{($w_iw_i^*\in N '$)}\\
&=\sum_{i=1}^d\tau(z^{-1/2}xz^{-1/2}(\tau\ten\id)((w_i^*\ten 1)e_N(w_i\ten 1)))\\
&=\tau(z^{-1/2}xz^{-1/2}(\tau\ten\id)(\om))\\
&=\tau(z^{-1/2}xz^{-1/2}z)\\
&=\tau(x).
\end{align*}
Thus, $T_i\circ\Psi_i = (\widetilde{\Psi})^{-1} $ (by equation (\ref{e:Psi})) is $\tau$-preserving, implying that $\Psi_i$ is $\tau$-preserving as $T_i$ is. Hence, the adjoint maps $\widetilde{\Psi}^*,\Psi_i^*,T_i^*\in\mc{CP}(N')$, defined relative to $\tau$, are all UCP, and satisfy $\widetilde{\Psi}^*\circ\Psi_i^*\circ T_i^*=\id_{N'}$. It follows that $\widetilde{\Psi}^*$ and $\Psi_i^*$ are $*$-automorphisms of $N'$, so there exist $u,v_1,...,v_d\in \mc{N}_M(N')=\mc{N}_M(N)$ such that 
$$\widetilde{\Psi}^*(x)=uxu^*, \ \ \ \textnormal{and} \ \ \ \Psi_i^*(x)=v_ixv_i^*, \ \ \ x\in N'.$$
Hence, $\Psi(x)=[M:N]\widetilde{\Psi}(z^{1/2}xz^{1/2})=[M:N]u^*z^{1/2}xz^{1/2}u$, so that
$$\Phi(x)=t\circ\Psi\circ t(x)=[M:N]u^tz^{1/2}xz^{1/2}\overline{u},$$
and therefore
$$\om=(\Phi\ten\id)(e_N)=[M:N](u^tz^{1/2}\ten 1)e_N(z^{1/2}\overline{u}\ten 1)=[M:N](1\ten z^{1/2}u)e_N(1\ten u^*z^{1/2}),$$
the last equality following from Lemma \ref{l:normalise} (and the entanglement of $e_N$).

Finally, as $uv_iu_i^*(\cdot)u_iv_i^*u^*=\id_{N'}$, it follows that 
$$F_i=(\Psi_i\ten\id)(e_N)=(v_i^*\ten1)e_N(v_i\ten 1)=(u_i^*u\ten1)e_N(u^*u_i\ten1),$$
and the proof is complete.
\end{proof}

\begin{remark} The hypothesis that $\tau|_{N'}$ is the Markov trace for $\bC\subseteq N'$ is valid whenever both $N$ and $N'$ are homogeneous subalgebras of $M_n(\C)$, that is $N=\bigoplus_{j=1}^J M_k(\C)\ten 1_l$ (where $k,l$ are constant in $j$). Indeed, in this case we have $n=Jkl$, $N'=\bigoplus_{j=1}^J1_{k}\ten M_{l}(\C)$, and $\mathrm{dim}N'=Jl^2$, so that $\frac{k}{n}=\frac{l}{\mathrm{dim}N'}$. Since the trace vector associated to $\tau|_{N'}$ is $\frac{1}{n}(k,...,k)$ and the Markov trace on $N'$ has trace vector $\frac{1}{\mathrm{dim}N'}(l,...,l)$, the claim follows.

Hence, Theorem \ref{t:Werner} applies in particular whenever $N'=\bigoplus_{j=1}^JM_{l}(\C)$ ($k\equiv 1$). As noted above, homogeneous subalgebras model a distinguished special case of  hybrid classical/quantum codes that lend themselves to explicit code constructions and analyses \cite{cao2021higher,grassl2017codes,nemec2018hybrid,nemec2021infinite}.  
\end{remark}

\begin{remark} Verdon recently generalised Werner's characterisation of tight teleportation schemes to the setting of entanglement-invertible channels using graphical techniques \cite{verdon}. One can phrase Theorem \ref{t:Werner} in Verdon's context, but it is unclear whether the explicit structure of our resulting scheme (i.e., unitary Pimsner-Popa basis in the normaliser) would follow from their characterisation. In any event, our independent work uses different techniques. 
\end{remark}

\section{Applications to Quantum Chromatic Numbers}

Quantum graphs can be studied from a variety of perspectives, including non-commutative confusability graphs of quantum channels \cite{dsw}, quantum relations \cite{weav1,weav2}, and $C^*$-algebras with a quantum adjacency matrix \cite{mrv,betal}. See \cite{daws} for a recent survey and relations between the approaches. In this work, we follow Weaver's approach \cite{weav1,weav2} via quantum relations, so that a \textit{quantum graph} is a triple $(\mc{S},M,\BH)$, consisting of a von Neumann algebra $M\subseteq\BH$ and a weak* closed operator system $\mc{S}\subseteq\BH$ which is an $M'$-bimodule. 

A simple way to construct quantum graphs over a von Neumann algebra $M$ is through inclusions: any von Neumann subalgebra $N\subseteq M$ gives rise to a pair of quantum graphs $(M,N',\BH)$ and $(N',M,\BH)$, whose associated bimodules are given by the inclusions $N\subseteq M$ and $M'\subseteq N'$, respectively. In this section we combine some of our techniques with those of \cite{bgh} to compute chromatic numbers for examples of such quantum graphs. For simplicity, we restrict attention to finite-dimensional examples and leave the infinite-dimensional generalizations to future work.

As with quantum graphs themselves, generalizations of graph theoretic parameters including chromatic numbers can be studied from a variety of perspectives. Motivated by \cite[Definition 5.10]{bhtt} and \cite[Theorem 4.7]{bgh}, we will use the following definition. See \cite[Theorem 4.7]{bgh} for the mentioned equivalence.

\begin{definition}\label{d:colour} A quantum graph $(\mc{S},M,\mc{B}(H))$ on a finite-dimensional Hilbert space $H$ \textit{is $(L,c)$ colourable}, where $L$ is a tracial von Neumann algebra and $c\in \N$, if the following two equivalent conditions hold:
\begin{enumerate}
\item There is a UCP map $\Phi:\ell^\infty_c\rightarrow M\ten L$
of the form $\Phi(\cdot)=\sum_{i=1}^m A_i^*(\cdot)A_i$ satisfying
\begin{enumerate}
\item $A_i(M'\ten 1_L)A_j^*\subseteq \ell^\infty_c$ for all $i,j$, and
\item $A_i(\mc{S}\cap(M')^\perp\ten 1_L)A_j^*\subseteq(\ell^\infty_c)^\perp,$ for all $i,j$.
\end{enumerate}
\item There is a PVM $\{P_a\}_{a=1}^c$ in $M\ten L$ satisfying 
\begin{equation}\label{e:col}P_a((\mc{S}\cap (M')^\perp)\ten 1_L)P_a=0, \ \ \ a=1,...,c.
\end{equation}
\end{enumerate}
Orthogonal complements of $M'$ and $\ell^\infty_c$ are taken in $\mc{B}(H)$ and $M_c(\C)$, respectively.

We will then refer to either $(L,c,\Phi)$ or $(L,c,\{P_a\}_{a=1}^c)$ as a \emph{colouring} of $(\mc{S},M,\mc{B}(H))$.
\end{definition}

\begin{remark} It is not clear to the authors that Definition \ref{d:colour} is independent of the embedding $M\subseteq\BH$.\end{remark}

By \cite[Theorem 4.7]{bgh}, Definition \ref{d:colour}(1) means precisely that $(L,c,\Phi)$ forms a perfect quantum commuting strategy for the quantum-to-classical graph homomorphism game between $(\mc{S},M,\mc{B}(H))$ and $(M_c(\C),\ell^\infty_c,M_c(\C))$ (the complete graph on $c$ vertices). Condition (a) is then viewed as a type of ``synchronicity'' condition arising from this non-local game. See \cite[\S4,\S5]{bgh} and \cite[\S5]{bhtt} for details and related notions. One can also interpret condition (a) in terms of quantum relations: it means that 
$\Phi$ preserves the ``diagonal'' quantum relations within the reflexive quantum relations defined by the quantum graphs (see \cite{weav2}). Condition (b) means that $(L,c,\Phi)$ is an ``entanglement assisted'' quantum graph homomorphism from $(\mc{S},M,\mc{B}(H))$ to $(M_c(\C),\ell^\infty_c,M_c(\C))$. Indeed, when $L=\C$, we recover the notion of quantum graph homomorphism through pushforwards of traceless operator systems as introduced by Stahlke \cite{s}. Note that the use of orthogonal complements in Definition \ref{d:colour} matches Stahlke's definition for traceless operator systems. For simplicity, we do not consider more general algebraic colourings as in \cite{bgh} (which loosens restrictions on the $*$-algebra $L$) but some of our arguments carry through verbatim to ``hereditary'' colourings (see \cite[\S5]{bgh} for details on algebraic colourings). 

Definition \ref{d:colour}(2) is a useful reformulation that will be frequently used in the sequel.

\begin{definition} Let $(\mc{S},M,\mc{B}(H))$ be a quantum graph on a finite-dimensional Hilbert space $H$. 
\begin{itemize}
\item Its \textit{quantum commuting chromatic number} is
$$\chi_{qc}(\mc{S},M,\mc{B}(H)):=\min\{c\in\N\mid \text{$(\mc{S},M,\mc{B}(H))$ is $(L,c)$ colourable for some $L$}\}.$$
\item Its \textit{quantum chromatic number} is
$$\chi_{q}(\mc{S},M,\mc{B}(H)):=\min\{c\in\N\mid \text{$(\mc{S},M,\mc{B}(H))$ is $(L,c)$ colourable with $L$ finite-dimensional}\}.$$
\item Its \textit{local chromatic number} is
$$\chi_{loc}(\mc{S},M,\mc{B}(H)):=\min\{c\in\N\mid \text{$(\mc{S},M,\mc{B}(H))$ is $(L,c)$ colourable with $L=\C$}\}.$$
\end{itemize}
We sometimes refer to the corresponding sets of colourings as \emph{$qc$-colourings}, \emph{$q$-colourings} and \emph{$loc$-colourings}, respectively.
\end{definition}

For complete quantum graphs $(M_n(\C),M,M_n(\C))$, it was shown in \cite[Theorem 5.6, Theorem 5.9]{bgh} that
$$\chi_q(M_n(\C),M,M_n(\C))=\chi_{qc}(M_n(\C),M,M_n(\C))=\mathrm{dim}M.$$
Their result, which utilizes teleportation type techniques for one direction, generalises in a straightforward fashion to quantum graphs from finite-dimensional inclusions $N\subseteq M$ with $N$ a factor, see Theorem \ref{t:chrom} below. We include details for convenience of the reader.

\begin{proposition}\label{p:chrom} Let $N\subseteq M$ be an inclusion of von Neumann algebras on a finite-dimensional Hilbert space $H$, with $N$ a factor. Then
\begin{equation}\label{e:chromatic}\chi_q(N',M, \mc{B}(H))\leq [M:N].\end{equation}
\end{proposition}

\begin{proof} Let $M=\bigoplus_{j=1}^m 1_{n_j}\ten M_{k_j}(\C)$ be the induced decomposition from the representation $M\subseteq\mc{B}(H)$. Since $N$ is a factor, we have $N\cong M_d(\C)$ for some $d$, and without loss of generality, for each $j$, there exists $l_j\in\N$ for which $k_j=l_jd$, so that $[M:N]=\sum_{j=1}^d l_j^2$, and
$$M=\bigoplus_{j=1}^m 1_{n_j}\ten M_{l_j}(\C)\ten M_d(\C)=\bigg(\bigoplus_{j=1}^m 1_{n_j}\ten M_{l_j}(\C)\bigg)\ten M_d(\C).$$
The embedding $N\subseteq M$ is then simply $x\mapsto 1_n\ten x$, where $n=\sum_{j=1}^m n_j l_j$. 
For each compressed inclusion $1_{l_j}\ten M_d(\C)\subseteq M_{l_j}(\C)\ten M_d(\C)$, $x\mapsto 1_{l_j}\ten x$ (still unital), pick an orthonormal Pimsner-Popa basis $\{u_i\}_{i=1}^{l_j^2}$ of unitaries lying in $M_{l_j}(\C)\ten 1$, and let $e_j\in M_{l_j}(\C)\ten M_{l_j}(\C)$ denote the Jones projection for the inclusion $\C\subseteq M_{l_j}(\C)$ (i.e., the maximally entangled state).

Set $l=\mathrm{lcm}(l_1,....,l_m)$, and for each $j$, pick a unital $*$-homomorphism $\pi_j:M_{l_j}(\C)\hookrightarrow M_l(\C)$. Letting $\Sigma$ denote the tensor flip, define projections $P_{i,j}\in M\ten M_l(\C)$ by 
$$P_{i,j}:= 1_{n_j}\ten(\id_{l_j}\ten \id_d\ten\pi_j)(\Sigma_{23}((u_i^*\ten1_{l_j})e_j(u_i\ten 1_{l_j})\ten 1_d)), \ \ \ j=1,...,m, \ i=1,...,l_j^2.$$
Then $\{P_{i,j}\}$ is a PVM:
\begin{align*}\sum_{j=1}^m\sum_{i=1}^{l_j^2}P_{i,j}&=\sum_{j=1}^m\sum_{i=1}^{l_j^2}1_{n_j}\ten(\id_{l_j}\ten \id_d\ten\pi_j)(\Sigma_{23}((u_i^*\ten1_{l_j})e_j(u_i\ten 1_{l_j})\ten 1_d))\\
&=\sum_{j=1}^m1_{n_j}\ten(\id_{l_j}\ten \id_d\ten\pi_j)(\Sigma_{23}(1_{l_j}\ten 1_{l_j}\ten 1_d))\\
&=\sum_{j=1}^m1_{n_j}\ten(\id_{l_j}\ten \id_d\ten\pi_j)(1_{l_j}\ten 1_d\ten 1_{l_j})\\
&=\sum_{j=1}^m1_{n_j}\ten1_{l_j}\ten 1_d \ten 1_l\\
&=1_M\ten 1_l.
\end{align*}

Now, the relative complement 
\begin{align*}N'\cap (M')^\perp&=\{X=[X_{j,j'}]\ten 1_d\in\mc{B}(\oplus_{j=1}^m \C^{n_j}\ten \C^{l_j})\ten 1_d\mid X \perp(\bigoplus_{j=1}^m M_{n_j}(\C)\ten 1_{l_j}\ten 1_d)\}\\
&=\{X=[X_{j,j'}]\ten 1_d\in\mc{B}(\oplus_{j=1}^m \C^{n_j}\ten \C^{l_j})\ten 1_d\mid (\id\ten\tau_{l_j})(X_{j,j})=0 \ \forall \ j=1,...,m\},
\end{align*}
where $\tau_{l_j}$ is the normalised trace on $M_{l_j}(\C)$. Since 
$$(1_{n_j}\ten e_j)(Y\ten 1_{l_j})(1_{n_j}\ten e_j)=(\id\ten\tau_{l_j})(Y)\ten e_j$$
for any $Y\in M_{n_j}(\C)\ten M_{l_j}(\C)$ (maximally entangled state is a trace vector), for any $X=[X_{j,j'}]\ten 1_d\in N'\cap (M')^\perp$, and every $j$ we have

\begin{align*}&(1_{n_j}\ten (u_i^*\ten 1_{l_j})e_j(u_i\ten 1_{l_j}))([X_{k,k'}]\ten 1_{l_j})(1_{n_j}\ten (u_i^*\ten 1_{l_j})e_j(u_i\ten 1_{l_j}))\\
&=(1_{n_j}\ten (u_i^*\ten 1_{l_j})e_j(1_{l_j}\ten u_i^t))(X_{j,j}\ten 1_{l_j})(1_{n_j}\ten (1_{l_j}\ten \overline{u_i})e_j(u_i\ten 1_{l_j}))\\
&=(1_{n_j}\ten (u_i^*\ten 1_{l_j})e_j)(X_{j,j}\ten 1_{l_j})(1_{n_j}\ten (1_{l_j}\ten e_j(u_i\ten 1_{l_j}))\\
&=(1_{n_j}\ten u_i^*\ten 1_{l_j})((\id_{n_j}\ten\tau_{l_j})(X_{j,j})\ten e_j)(1_{n_j}\ten (1_{l_j}\ten u_i\ten 1_{l_j}))\\
&=0.
\end{align*}
Simple manipulations with the flip map $\Sigma_{23}$ show that
$$P_{i,j}(X\ten 1_l)P_{i,j}=0, \ \ \ X\in N'\cap(M')^\perp.$$
Hence, $(M_l(\C),[M:N],\{P_{i,j}\})$ is a finite-dimensional colouring of $(N',M,\BH)$, so its quantum chromatic number is at most $[M:N]$.
\end{proof}

Continuing with the proof strategy of \cite[Theorem 5.9]{bgh}, we now show that equality holds in (\ref{e:chromatic}). The same argument works more generally for hereditary colourings (see \cite[\S 5]{bgh}), but for simplicity of presentation we restrict to quantum and quantum commuting colourings. We require a generalised version of \cite[Lemma 5.8]{bgh}.

\begin{lemma}\label{l:chrom} Let $N\subseteq M$ be factors on a finite-dimensional Hilbert space $H$.
  If $(L,c,\{P_a\}_{a=1}^c)$ is a colouring of $(N',M,\mc{B}(H))$,
  then for each $a$, $R_a=[M:N](E_N\ten\id_L)P_a$ is a projection such that $\sum_{a=1}^c R_a=[M:N]1_N\ten 1_L$. 
\end{lemma}

\begin{proof} There exist $d,m,n\in\N$ such that $H=\C^d\ten\C^m\ten\C^n$, $M=1_d\ten M_m(\C)\ten M_n(\C)$ and $N=1_d\ten1_m\ten M_n(\C)$. Then $[M:N]=m^2$. 

Since $P_a\in M\ten L=1_d\ten M_m(\C)\ten M_n(\C)\ten L$, write
$$P_a=\sum_{i,j=1}^m\sum_{k,l=1}^n 1_d\ten e_{i,j}\ten e_{k,l}\ten P^a_{i,j,k,l},$$
where $e_{i,j}\in M_m(\C)$ and $e_{k,l}\in M_n(\C)$ are matrix units and $P^a_{i,j,k,l}\in L$.
Let $i_0,j_0\in\{1,...,m\}$, $i_0\neq j_0$. Then 
$$1_d\ten e_{i_0,j_0}\ten 1_n\in N'\cap(M')^\perp = \{X\in M_d(\C)\ten M_m(\C)\ten 1_n\mid (\id_d\ten\tau_{m}\ten\tau_n)(X)=0\}.$$
Hence,
\begin{align*}0&=P_a(1_d\ten e_{i_0,j_0}\ten 1_n\ten 1_L)P_a\\
&=\sum_{i,i',j,j'=1}^m\sum_{k,k',l,l'=1}^n 1_d\ten e_{i,j}e_{i_0,j_0}e_{i',j'}\ten e_{k,l}e_{k',l'}\ten P^a_{i,j,k,l}P^a_{i',j',k',l'}\\
&=\sum_{i,j'=1}^m\sum_{k,k',l'=1}^n 1_d\ten e_{i,j'}\ten e_{k,l'}\ten P^a_{i,i_0,k,k'}P^a_{j_0,j',k',l'}\\
&=\sum_{i,j'=1}^m\sum_{k,l'=1}^n 1_d\ten e_{i,j'}\ten e_{k,l'}\ten \bigg(\sum_{k'=1}^nP^a_{i,i_0,k,k'}P^a_{j_0,j',k',l'}\bigg),
\end{align*}
so that 
\begin{equation}\label{e:star}\sum_{k'=1}^nP^a_{i,i_0,k,k'}P^a_{j_0,j',k',l'}=0, \ \ \ \forall \ i,j',k,l', \ \ i_0\neq j_0.\end{equation}
Similarly, $1_d\ten(e_{i_0,i_0}-e_{j_0,j_0})\ten1_n\in N'\cap(M')^\perp$ and it follows that
\begin{equation}\label{e:star2}\sum_{l=1}^nP^a_{i,i_0,k,l}P^a_{i_0,j',l,l'}=\sum_{l=1}^nP^a_{i,j_0,k,l}P^a_{j_0,j',l,l'}, \ \ \ \forall \ i,j',k,l',i_0,j_0.\end{equation}
Finally, since $P_i$ is a projection, one easily sees that
\begin{equation}\label{e:star3}P^a_{i,j',k,l'}=\sum_{j=1}^m\sum_{l=1}^nP^a_{i,j,k,l}P^a_{j,j',l,l'} \ \ \ \forall \ i,j',k,l'.\end{equation}

The conditional expectation $E_N:M\rightarrow N$ is the (normalised) partial trace $(\id_d\ten\tau_m\ten\id_n)$, so that
$$R_a=m^2(E_N\ten\id_L)P_a=m\sum_{i=1}^m\sum_{k,l=1}^n 1_d\ten e_{k,l}\ten P^a_{i,i,k,l}.$$
Hence, 
\begin{align*}R_a^2&=m^2\sum_{k,l'=1}^n 1_d\ten e_{k,l'}\ten \bigg(\sum_{i,i'=1}^m\sum_{k'=1}^nP^a_{i,i,k,k'}P^a_{i',i',k',l'}\bigg)\\
&=m^2\sum_{k,l'=1}^n 1_d\ten e_{k,l'}\ten \bigg(\sum_{i=1}^m\sum_{k'=1}^nP^a_{i,i,k,k'}P^a_{i,i,k',l'}\bigg) \ \ \ \ \textnormal{(by (\ref{e:star}))}\\
&=m\sum_{k,l'=1}^n 1_d\ten e_{k,l'}\ten \bigg(\sum_{i,j=1}^m\sum_{k'=1}^nP^a_{i,j,k,k'}P^a_{j,i,k',l'}\bigg) \ \ \ \ \textnormal{(by (\ref{e:star2}))}\\
&=m\sum_{k,l'=1}^n 1_d\ten e_{k,l'}\ten \bigg(\sum_{i=1}^mP^a_{i,i,k,l'}\bigg) \ \ \ \ \textnormal{(by (\ref{e:star3}))}\\
&=R_a.
\end{align*}
That $R_a=R_a^*$ and $\sum_{a=1}^cR_a=m^21_N\ten 1_L$ are immediate from its definition.
\end{proof}

\begin{remark}\label{r:chrom} Note that in the above proof we took $R_a\in 1_d\ten M_n(\C)\ten L$. Since $P_a\in 1_d\ten M_m(\C)\ten M_n(\C)\ten L$, we could also trace out the (trivial) first leg of $P_a$ in the definition of $R_a$ to the same end. This will be done in the next proof.
\end{remark}

\begin{theorem}\label{t:chrom} Let $N\subseteq M$ be an inclusion of von Neumann algebras on a finite-dimensional Hilbert space $H$ with $N$ a factor. Then 
$$\chi_q(N',M, \mc{B}(H))=\chi_{qc}(N',M,\mc{B}(H))=[M:N].$$
\end{theorem}

\begin{proof} By Proposition \ref{p:chrom} and the fact that $\chi_q(N',M, \mc{B}(H))\geq\chi_{qc}(N',M,\mc{B}(H))$, it suffices to show that
$$\chi_{qc}(N',M,\mc{B}(H))\geq [M:N].$$
Let $M=\bigoplus_{j=1}^m 1_{n_j}\ten M_{k_j}(\C)$ be the induced decomposition from the representation $M\subseteq\mc{B}(H)$. As in the proof of Proposition \ref{p:chrom}, we may assume $N\cong M_d(\C)$ for some $d\in\N$ and $k_j=l_jd$ for some $l_j\in\N$  so that 
$$M=\bigoplus_{j=1}^m 1_{n_j}\ten M_{l_j}(\C)\ten M_d(\C)=\bigg(\bigoplus_{j=1}^m 1_{n_j}\ten M_{l_j}(\C)\bigg)\ten M_d(\C).$$
Note that $[M:N]=\sum_{j=1}^ml_j^2$. 

Suppose $\{P_a\}_{a=1}^c\subseteq M\ten L$ is a $qc$-colouring of $(N',M,\mc{B}(H))$. Letting $z_j$ denote the central projection of $M$ onto the $j^{th}$ summand $M_{l_j}(\C)\ten M_d(\C)$, it follows that $\{P_a(z_j\ten 1_L)\}_{a=1}^c\subseteq Mz_j\ten L$ is a $qc$-colouring of $(z_jN'z_j,Mz_j,\mc{B}(z_jH))$, as $(Mz_j)'=M'z_j$ and $X\in z_jN'z_j\cap(M'z_j)^\perp\cap \mc{B}(z_jH)$, implies $X=z_jXz_j\in N'\cap (M')^\perp$, so that
$$P_a(z_j\ten 1_L)((z_jN'z_j\cap(M'z_j)^\perp\cap\mc{B}(z_jH))\ten 1_L)(z_j\ten 1_L)P_a=0.$$
By Lemma \ref{l:chrom}, $R^j_a=l_j^2(\tau_{n_j}\ten\tau_{l_j}\ten\id_d\ten\id_L)((z_j\ten 1)P_a)$ is a projection in $M_d(\C)\ten L$ satisfying $\sum_{a=1}^cR^j_a=l_j^2(1_d\ten 1_L)$ (see also Remark \ref{r:chrom}). Moreover, $R_a^i\perp R_a^j$ when $i\neq j$. To see this, first note that $N'=\mc{B}(\oplus_j \C^{n_j}\ten \C^{l_j})\ten 1_d$, so taking a matrix unit $(e^i_{i_0}\ten e^i_{j_0})(e^j_{k_0}\ten e^j_{l_0})^*\ten 1_d\in N'\cap (M')^\perp$ ($i\neq j$), and writing  
$$P_a(z_j\ten 1_L)=\sum_{k,l=1}^{l_j} 1_{n_j}\ten e^j_{k,l}\ten P^j_{a,k,l},$$
with $P^j_{a,k,l}\in M_d(\C)\ten L$, we have
\begin{align*}
0&=P_a((e^i_{i_0}\ten e^i_{j_0})(e^j_{k_0}\ten e^j_{l_0})^*\ten 1_d\ten 1_L)P_a\\
&=P_a(z_i\ten 1_L)((e^i_{i_0}\ten e^i_{j_0})(e^j_{k_0}\ten e^j_{l_0})^*\ten 1_d\ten 1_L)(z_j\ten 1_L)P_a\\
&=e^i_{i_0}(e^j_{k_0})^*\ten\sum_{k,l=1}^{l_i}\sum_{k',l'=1}^{l_j} e^i_{k,l}e^i_{i_0}(e^j_{l_0})^*e^j_{k',l'}\ten P^i_{a,k,l}P^j_{a,k',l'}\\
&=e^i_{i_0}(e^j_{k_0})^*\ten\sum_{k=1}^{l_i}\sum_{l'=1}^{l_j} e^i_{k}(e^j_{l'})^*\ten P^i_{a,k,i_0}P^j_{a,l_0,l'}.
\end{align*}
Hence, $P^i_{a,k,i_0}P^j_{a,l_0,l'}=0$ for all $k,i_0,l_0,l'$ whenever $i\neq j$, from which the claim $R^i_a\perp R^j_a$ is easily deduced. It follows that $R_a:=\sum_{j=1}^m R^j_a$ is a projection in $M_d(\C)\ten L$ satisfying
$$\sum_{a=1}^c R_a=\sum_{a=1}^c\sum_{j=1}^m R^j_a=\sum_{j=1}^ml_j^2(1_d\ten 1_L)=[M:N]1_d\ten L_1.$$
On the other hand, $1_d\ten 1_L-R_a\geq 0$ so that
$$(c-[M:N])1_d\ten 1_L=\sum_{a=1}^c(1_d\ten 1_L-R_a)\geq0,$$
implying $c\geq [M:N]$. 
\end{proof}

We now combine some of our techniques from previous sections with those of \cite[Theorem 5.9]{bgh} to calculate $\chi_{loc}(M,N',\BLTM)$ for a large class of finite-dimensional inclusions $N\subseteq M$.

\begin{theorem}\label{t:finite} Let $N\subseteq M$ be a strongly Markov inclusion of finite-dimensional von Neumann algebras which admits an orthonormal Pimsner-Popa basis $\{u_i\}$ for $M$ over $N$ in $\mc{N}_M(N)$. Then 
$$\chi_{loc}(M,N',\BLTM)=\chi_{q}(M,N',\BLTM)=\chi_{qc}(M,N',\BLTM)=[M:N].$$ 
\end{theorem}

\begin{proof} First, by the Markov property, the cardinality of $\{u_i\}$ is the index $[M:N]$:
$$1=\tau_1(1)=\sum_i\tau_1(u_i^*e_Nu_i)=\sum_i\tau_1(e_N)=\frac{|\{u_i\}|}{[M:N]}.$$
It suffices to show $\chi_{loc}(M,N',\BLTM)\leq[M:N]$ and $\chi_{qc}(M,N',\BLTM)\geq[M:N]$. 

$\underline{\chi_{loc}(M,N',\BLTM)\leq[M:N]}:$ 
Let $\tau_d$ denote the unique tracial state on $\BLTM$.
Recall that $\mc{N}_M(N)=\mc{N}_{M}(E_N)$ (Lemma \ref{l:normalise}), so that each $u_i$ normalises $E_N$, and that $\{u_i^*\psi_n\mid i=1,...,[M:N], \ n=1,...,\dim N\}$ forms an orthonormal basis of $L^2(M,\tau)$ whenever $\{\psi_n\}$ is an orthonormal basis of $L^2(N,\tau)$ (see the proof of Lemma \ref{l:ONB}). It follows that $E_N$ is $\tau_d$-invariant: given $x\in M$, we have
\begin{align*}\tau_d(x)&=\frac{1}{\dim M}\sum_{n=1}^{\dim N}\sum_{i=1}^{[M:N]}\la xu_i^*\psi_n,u_i^*\psi_n\ra\\
&=\frac{1}{\dim M}\sum_{n=1}^{\dim N}\sum_{i=1}^{[M:N]}\la e_Nu_ixu_i^*e_N\psi_n,\psi_n\ra\\
&=\frac{1}{\dim M}\sum_{n=1}^{\dim N}\sum_{i=1}^{[M:N]}\la E_N(u_ixu_i^*)\psi_n,\psi_n\ra\\
&=\frac{1}{\dim M}\sum_{n=1}^{\dim N}\sum_{i=1}^{[M:N]}\la u_iE_N(x)u_i^*\psi_n,\psi_n\ra\\
&=\tau_d(E_N(x)).
\end{align*}
Combined with the faithfulness of $\tau_d$, it follows that
$$M\cap N^\perp=\{x\in M\mid \tau_{d}(xy)=0, \ \forall \ y\in N\}=\{x\in M\mid E_N(x)=0\}=\mathrm{Ker}(E_N).$$

Now, let $P_i:=u_i^*e_Nu_i$. Then $\{P_i\}_{i=1}^{[M:N]}$ is a PVM in $N'$ since $e_N\in N'$, and each $u_i$ normalises $N'$. 
To show that $\{P_i\}_{i=1}^{[M:N]}$ is an $loc$-colouring of the quantum graph $(M,N',\BLTM)$, by (\ref{e:col}), it suffices to show that $P_ixP_i=0$ for all $x$ in  $M\cap ((N')')^\perp=M\cap N^\perp$. But for any $x\in M\cap N^\perp=\mathrm{Ker}(E_N)$,
\begin{align*}P_ixP_i&=(u_i^*e_Nu_i)x(u_i^*e_Nu_i)\\
&=u_i^*e_N(u_i x u_i^*)e_Nu_i\\
&=u_i^*E_N(u_i x u_i^*)e_Nu_i\\
&=u_i^*u_iE_N(x)u_i^*e_Nu_i\\
&=0.
\end{align*}
$\underline{\chi_{qc}(M,N',\BLTM)\geq[M:N]}:$ Suppose $\{P_a\}_{a=1}^c\subseteq N'\ten L$ is a $qc$-colouring with $L$ a tracial von Neumann algebra. Then $P_a(x\ten 1_L)P_a=0$ for all $x\in M\cap N^\perp=\mathrm{Ker}(E_N)$, implying $P_a(x\ten 1_L)P_a=P_a(E_N(x)\ten 1_L)P_a$ for all $x\in M$. In particular,
$$P_a(u_iu_j^*\ten 1_L)P_a=P_a(E_N(u_iu_j^*)\ten 1_L)P_a=\delta_{i,j}P_a.$$
Hence, for each $a$, $\{(u_i^*\ten 1_L)P_a(u_i\ten 1_L)\}_{i=1}^{[M:N]}$ is a family of mutually orthogonal projections in $N'\ten L$. By (the left basis version of) \cite[Proposition 2.24]{jp}, the map
$$E_{M'}:N'\ni y\mapsto\frac{1}{[M:N]}\sum_{i=1}^{[M:N]}u_i^*yu_i\in M'$$
is a conditional expectation (unique with respect to canonical traces on $N'$ and $M'$). Define $R_a:=[M:N](E_{M'}\ten \id_L)P_a$. Then
$$R_a=\sum_{i=1}^{[M:N]}(u_i^*\ten 1_L)P_a(u_i\ten 1_L)$$
is a projection in $M'\ten L$ satisfying $\sum_{a=1}^cR_a=[M:N]1_{M'}\ten 1_L$. But then, 
$$(c-[M:N])1_{M'}\ten 1_L=\sum_{a=1}^c(1_{M'}\ten 1_L-R_a)\geq	0,$$
forcing $c\geq [M:N]$. 
\end{proof}

\begin{remark} Note that Theorem \ref{t:finite} does not contradict \cite[Theorem 5.11]{bgh} which forbids finite local chromatic number for complete quantum graphs of the form $(M_n(\C),M,M_n(\C))$ with $M$ non-abelian. There, the operator system $\mc{S}=M_n(\C)$ coincides with the algebra of bounded operators on the representation space $\C^n$, whereas in Theorem \ref{t:finite}, the operator system $\mc{S}=M$ is not the full algebra $\BLTM$ (unless $M=\C$). 
\end{remark}

\section{Outlook}

In this work, we introduced a model of quantum teleportation in the commuting operator framework, deepened connections with subfactor theory and generalised Werner's characterisation of tight teleportation schemes. Several natural lines of investigation are left for future work, including
\begin{enumerate}
\item rigidity of teleportation for more general inclusions $N\subseteq M$;
\item futher connections with subfactor theory, depth-2 inclusions \cite{nv}, weak $C^*$-Hopf algebras \cite{metal,metal2,nv}, and categorical approaches to quantum teleportation \cite{ac,jlw,lwj}; 
\item superdense coding in the commuting operator framework, building on \cite{huang};
\item colourings of infinite quantum graphs from finite-index inclusions of II$_1$ factors;
\item connections with quantum automorphism groups \cite[\S7]{bevw}.
\end{enumerate}

\vspace{0.1in}

{\noindent}{\it Acknowledgements.} The second author was partially supported by the NSERC Discovery Grant RGPIN-2017-06275, and would like to thank Michael Brannan, Samuel Harris and Ivan Todorov for helpful discussions. The third author was partially supported by the NSERC Discovery Grant RGPIN-2018-400160, and would like to thank Rajesh Pereira for helpful discussions. The authors would also like to thank the anonymous referees for helpful comments which improved the presentation of the paper.

\end{document}